\documentclass[12pt]{amsart}
\usepackage[margin=2.3cm]{geometry}

\usepackage{amsthm} 

\usepackage{hyperref}
\usepackage{amssymb,mathrsfs,mathcomp,amsfonts,longtable, hhline, textcomp, upgreek}
\usepackage[matrix,arrow,curve]{xy}
\usepackage{tabularx,multirow}
\usepackage{float}
\usepackage{url}

\newcommand{\mumu}{{\boldsymbol{\mu}}}
\newcommand{\B}{{\mathbf{B}}}

\newcommand{\QQ}{\mathbb{Q}}

\newcommand{\ZZ}{\mathbb{Z}}
\newcommand{\PP}{\mathbb{P}}
\newcommand{\OOO}{{\mathscr{O}}} 
 
\newcommand{\MMM}{{\mathscr{M}}} 
 
\newcommand{\BBB}{{\mathscr{B}}}

\newcommand{\fracf}[2]{#1/#2}

\newcommand{\h}{\operatorname{h}}
\newcommand{\p}{\operatorname{p}}
\newcommand{\g}{\operatorname{g}} 

\newcommand{\qq}{\mathbin{\sim_{\scriptscriptstyle{\mathbb{Q}}} } }

\newcommand{\qW}{\operatorname{q}_{\operatorname{W}}}
\newcommand{\qQ}{\operatorname{q_{\QQ}}}
\newcommand{\Sing}{\operatorname{Sing}}

\newcommand{\Bs}{\operatorname{Bs}}

\newcommand{\Pic}{\operatorname{Pic}}
\newcommand{\Cl}{\operatorname{Cl}}
\newcommand{\Clt}[1]{\operatorname{Cl}(#1)_{\mathrm {t}}}

\newcommand{\ct}{\operatorname{ct}}

\newcommand{\type}[1]{$\mathrm{#1}$} 

\newcommand{\xref}[1]{{\rm~\ref{#1}}}

\makeatletter
\@addtoreset{equation}{subsection}
\makeatother

\theoremstyle{definition}
\newtheorem{step}{Step}
\swapnumbers
\theoremstyle{plain}
\newtheorem{theorem}[subsection]{Theorem}
\newtheorem{lemma}[subsection]{Lemma}

\newtheorem{proposition}[subsection]{Proposition}
\newtheorem{stheorem}[equation]{Theorem}

\newtheorem{corollary}[subsection]{Corollary}
\newtheorem{scorollary}[equation]{Corollary}
\newtheorem*{claim*}{Claim}
\newtheorem{sclaim}[equation]{Claim}
\newtheorem{slemma}[equation]{Lemma}
\newtheorem{sproposition}[equation]{Proposition}

\newtheorem{sproposition-definition}[equation]{Proposition-Definition}

\theoremstyle{definition}
\newtheorem{setup}[subsection]{Set-up}

\newtheorem*{definition*}{Definition}
\newtheorem{sdefinition}[equation]{Definition}
\newtheorem{example-remark}[subsection]{Remark-Example}
\newtheorem{subexample-remark}[equation]{Remark-Example}

\newtheorem*{notation*}{Notation}

\newtheorem{sremark}[equation]{Remark} 
\newtheorem{sconvention}[equation]{Convention}

\newcounter{NN}\numberwithin{NN}{section}
\renewcommand{\theNN}{\arabic{NN}${}^o$}
\def\nr{\refstepcounter{NN}{\theNN}}%

\renewcommand{\theenumi}{\rm (\arabic{enumi})}
\renewcommand\labelenumi{\rm (\arabic{enumi})}

\renewcommand\labelenumi{\rm (\roman{enumi})}
\renewcommand\theenumi{\rm (\roman{enumi})}

\begin{document}
\title{On the birational geometry of $\QQ$-Fano threefolds \\ of large Fano index, I} 

\address{ 
Steklov Mathematical Institute of Russian Academy of Sciences, Moscow, Russian Federation
} 
\email{prokhoro@mi-ras.ru}

\author{Yuri~Prokhorov}
\thanks{This work is supported by the Russian Science Foundation under grant no. 23-11-00033, 
\url{https://rscf.ru/project/23-11-00033/}}

\begin{abstract}
We investigate the rationality problem for $\QQ$-Fano threefolds of Fano index $\ge 2$. 
\end{abstract}

\maketitle

\section{Introduction}

A three-dimensional algebraic projective variety $X$ 
is called \textit{$\QQ$-Fano threefold} if it has only terminal 
$\QQ$-factorial singularities, $\Pic(X)\simeq \ZZ$,
and its anticanonical divisor $-K_X$ is ample.
The class of these varieties is important in birational geometry because it is one of the possible outputs 
of the Minimal Model Program in dimension $3$. 
It is known that $\QQ$-Fano threefolds are bounded, i.e. they lie in a finite number of algebraic 
families. Moreover, the methods of \cite{Kawamata:bF} allow to obtain a (huge) list 
of numerical invariants of $\QQ$-Fano threefolds \cite{GRD}. 
At the moment there is no classification, but there are a lot of partial results. 

This work is a sequel to our previous papers \cite{P:2019:rat:Q-Fano}, \cite{P:fano-conic}. 
We are interested in the birational geometry of $\QQ$-Fano threefolds rather than 
biregular one.
Mainly, we will discuss the rationality question. 

The \textit{$\QQ$-Fano index} of a $\QQ$-Fano threefold $X$ is the maximal integer $\qQ(X)$ that divides the
canonical class $K_X$ in the Weil divisor class group modulo torsion (see~\eqref{eq:deq-q}).
A Weil divisor $A$ such that $-K_X\qq \qQ(X) A$ we call the \textit{fundamental divisor} 
and denote it by $A_X$. It turns out that the classification of $\QQ$-Fano threefolds of large index $\qQ(X)$
is much simpler (see \cite{Suzuki-2004} and \cite{P:2010:QFano}). Moreover, $\QQ$-Fano threefolds of large index $\qQ(X)$
are expected to be rational: 

\begin{theorem}[{\cite{P:2019:rat:Q-Fano}}]
\label{thm:q8}
Let $X$ be a $\QQ$-Fano threefold. 
If $\qQ\ge 8$, then $X$ is rational.
\end{theorem} 

On the other hand, there are nonrational $\QQ$-Fano threefolds of large 
index.
For example, T.~Okada \cite{Okada2019} showed that there are $\QQ$-Fano 
threefold hypersurfaces of index 
$2$, $3$, $5$, and~$7$ that
are not rational. 

The following  invariant will be very important in the sequel:
\begin{equation*}
\label{def:pn}
\p_n(X):=\max \left\{ \h^0(X,\OOO_X(D)) \mid D\qq nA_X\right\}. 
\end{equation*} 
If the Weil divisor class group $\Cl(X)$ is torsion free, then the above definition becomes simpler:
\[
\p_n(X)=\h^0(X,\OOO_X(nA_X)).
\] 

Our main result is as follows:

\begin{theorem}
Let $X$ be a $\QQ$-Fano threefold with $\qQ(X)\ge 2$.
If one of the following conditions hold, then $X$ is rational
\begin{enumerate}
\item
$\p_1(X)\ge4$,
\item
$\qQ(X)\ge 3$ and $\p_1(X)\ge3$,
\item
$\qQ(X)\ge 4$ and $\p_1(X)\ge2$,

\item
$\qQ(X)\ge 5$ and $\p_2(X)\ge2$ and $X$ is not of type \cite[\# 41422]{GRD}
\textup(see Proposition~\xref{prop5a}\textup),

\item
$\qQ(X)\ge 6$ and $\p_3(X)\ge2$.
\end{enumerate}
\end{theorem}

We also study birational geometry of $\QQ$-Fano threefolds with 
$\p_1(X)\ge 2$ and $\qQ(X)=2$ or $3$ 
(see Propositions~\ref{prop:3} and 
\ref{prop:2}).

\section{Preliminaries}

\subsection{Notation.}
We employ the following notation.
\begin{itemize}
\item 
$\sim$ and $\qq$ denote the linear and $\QQ$-linear equivalences of divisors, respectively; 
\item 
$\Cl(X)$ denotes the Weil divisor class group of a normal variety $X$;
\item
$\Clt{X}$ is the torsion subgroup of $\Cl(X)$;

\item
$\g(X)$ is the genus of the $\QQ$-Fano threefold $X$, that is, $\g(X):=\h^0(X, \OOO_X(-K_X))-2$;
\item
$\PP(w_1,\dots,w_n)$ is the weighted projective space with weights $w_1,\dots,w_n$; 
coinciding weights we group together as follows $\PP(\underbrace{w_1,\dots,w_1}_k,w_2, \dots )=\PP(w_1^k,w_2,\dots)$;
\item 
$\B(X)$ is the basket of singularities of a terminal threefold $X$ (see \cite{Reid:YPG});
this is a collection of virtual cyclic quotient singularities $\frac{1}{r_P}(1,-1,b_P)$
associated to each actual singular point of $X$. For short, in $\B(X)$ we list only indices of
these virtual points, i.e. $\B(X)=(\{r_P\})$.
\item 
$\mumu_N$ denotes the multiplicative cyclic group of order $N$. If $\mumu_N$ acts on $\mathbb{A}^n$ by
\[
(x_1,\dots, x_n) \longmapsto (\zeta_n^{w_1} x_1,\dots, \zeta_n^{w_n} x_n),
\]
where $\zeta_n$ is a primitive $N$-th root of $1$, then we say that $(w_1,\dots,w_n)$ are the weights of the action an we write 
$\mumu_N(w_1,\dots,w_n)$ to specify the action.
\end{itemize}

\subsection{$\QQ$-Fano threefolds}
For a Fano variety with at worst log terminal singularities we define \textit{Fano-Weil} and \textit{$\QQ$-Fano indices}
as follows:
\begin{equation}
\label{eq:deq-q}
\begin{array}{lll}
\qW(X)&:=&\max \{ q\in\ZZ \mid -K_X\sim q A,\ \text{$A$ is a Weil divisor}\},
\\[5pt]
\qQ(X)&:=&\max \{ q\in \ZZ \mid -K_X\qq q A,\ \text{$A$ is a Weil divisor}\}.
\end{array}
\end{equation} 
The \textit{fundamental divisor} of $X$ is a Weil divisor $A_X$ such that 
\begin{equation}
\label{eq:def-AX}
-K_X\qq \qQ(X) A_X.
\end{equation} 
Note that if $\Clt{X}\neq 0$, then the class of $A_X$
is not uniquely defined modulo linear equivalence. 
However, in the case $\qQ(X)=\qW(X)$ we always accept the following.
\begin{sconvention}
\label{conv:def-AXa}
If $\qQ(X)=\qW(X)$ we take $A_X$
so that 
\begin{equation*}
\label{eq:def-AXa}
-K_X\sim \qW(X)A_X. 
\end{equation*} 
\end{sconvention} 
The \textit{Hilbert series} of a $\mathbb{Q}$-Fano threefold $X$ is the formal power series \cite{Altinok2002}
\[
\h_X(t)= \sum_{m\ge 0} \h^0 (X, mA_X )\cdot t^m.
\]
It is computed by using the orbifold Riemann-Roch formula \cite{Reid:YPG}. If the group $\Cl(X)$ contains an element $T$ of $N$-torsion, we define \emph{$T$-Hilbert series} $\h_X(t,\sigma)\in \ZZ[[t,\sigma]]/(\sigma^N-1)$
as follows:
\[
\h_X(t,\sigma)= \sum_{m\ge 0} \sum_{j=0}^{N-1}\h^0 (X, mA_X+jT )\cdot t^m\sigma^j.
\]
Obviously, the above definition depends on the choice the class of $A_X$ in 
$\Cl(X)$. Typically calculating $\h_X(t)$ or $\h_X(t,\sigma)$ for our purposes we need only 
a few initial terms of the series.

Recall that the \textit{Gorenstein index} of a normal $\QQ$-Gorenstein variety $X$ is a minimal positive integer $r$ 
such that the Weil divisor $rK_X$ is Cartier.
\begin{stheorem}[{\cite{Suzuki-2004},} {\cite{P:2010:QFano}}]
\label{thm:index}
Let $X$ be a Fano threefold with terminal singularities
and let $r$ be its global Gorenstein index.
Then the following assertions hold:
\begin{enumerate}
\item\label{thm:indexa}
$\qW(X)$ divides $\qQ(X)$;
\item\label{thm:indexb}
$\qW(X)=\qQ(X)$ if and only if $r$ and $\qW(X)$ are coprime;
\item\label{thm:indexc}
$rA_X^3$ is an integer;
\item\label{thm:indexd}
$\qW(X),\ \qQ(X)\in \{1,\dots,9, 11, 13,17,19\}$.
\end{enumerate}
\end{stheorem}

\begin{stheorem}[{\cite{CampanaFlenner}}, {\cite{Sano-1996}}]
\label{thm:QFanoF}
Let $X$ be a Fano threefold with terminal singularities.
Assume that there exists a Cartier divisor $H$ on $X$ such that
$H\qq mA_X$, where $m<\qQ(X)$. Then the general member $S\in |H|$ is a smooth del Pezzo surface, the group $\Cl(X)$ is torsion free,
and $(X,H)$ is described by the following table.
The general member $X$ of each family is a $\QQ$-Fano threefold.
\end{stheorem} 
\begin{table}[H]
\setcounter{NN}{0}
\renewcommand{\arraystretch}{1.2}
\begin{tabularx}{0.97\textwidth}{|l|l|l|l|X|l|l|c|l|}
\hline
&$\qQ$&$\B(X)$ &$A_X^3$&\centering $X$ &$\g(X)$&\centering Rat? &$m$ & $K_S^2$
\\\hline
\nr \label{thm:QFanoF:1} & $4$ &$\varnothing$&$1$&$\PP^3$ &$33$& R &$1,\, 2,\, 3$ & $9,\, 8,\, 3$
\\\hline
\nr \label{thm:QFanoF:2} & $3$& $\varnothing$ &$2$&$X_2 \subset \PP^4$ & $28$& R &$1,\, 2$ &$8,\, 4$
\\\hline
\nr \label{thm:QFanoF:3} & $2$ &$\varnothing$ &$d$&del Pezzo threefold of degree $d\le 5$ & $4d+1$&& $1$ & $d$
\\\hline
\nr \label{thm:QFanoF:4} & $5$& $(2)$ &$1/2$&$\PP (1^3, 2)$ & $32$&R &$2,\, 4$&$9,\, 2$
\\\hline
\nr \label{thm:QFanoF:5} & $7$& $(2,3)$ & $1/6$&$\PP(1^2,2,3)$ & $29$&R &$6$ &$1$
\\\hline
\nr \label{thm:QFanoF:6} & $4$& $(3^2)$ &$1/3$&$X_6 \subset \PP(1^2,2,3^2)$ & $11$&R &$3$ &$1$
\\\hline
\nr \label{thm:QFanoF:7} & $5$& $(2,4)$ &$1/4$&$X_6 \subset \PP(1^2,2,3,4)$ & $16$ &R &$4$ &$1$
\\\hline
\nr \label{thm:QFanoF:8} & $6$& $(5)$ &$1/5$&$X_6 \subset \PP(1^2,2,3,5)$ &$22$& R &$5$ &$1$
\\\hline
\nr \label{thm:QFanoF:9} & $3$ &$(2^3)$&$1/2$&$X_6 \subset \PP (1^2, 2^2, 3)$ &$7$& &$2$ &$1$
\\\hline
\nr \label{thm:QFanoF:10} & $3$&$(2^2)$&$1$&$X_4 \subset \PP(1^3, 2^2)$ &$14$& R &$2$&$2$
\\\hline
\nr \label{thm:QFanoF:11} & $4$&$(3)$&$2/3$&$X_4 \subset \PP(1^3, 2, 3)$ & $22$&R &$3$&$2$
\\\hline
\nr \label{thm:QFanoF:12} & $3$&$(2)$&$3/2$&$X_3 \subset \PP (1^4, 2)$ &$21$& R &$2$ &$3$
\\\hline
\end{tabularx} 
\end{table}
The column ``Rat'' indicates the rationality of $X$.
One can see that almost all the $\QQ$-Fano threefolds in the table are rational.
In the case~\ref{thm:QFanoF:9} it is known that a very general variety in the family is 
not rational (and even not stably rational) \cite{Okada2019}. 
Rationality question of 
del Pezzo threefolds (case~\ref{thm:QFanoF:3}) has a long story. We refer to the book \cite{IP99} for references
and to \cite{Prz-Ch-Shr:DS}
for a detailed discussion of the case $d=2$. 

\subsection{Singularities}

For the classification of terminal threefold singularities we refer to \cite{Reid:YPG}. 

\begin{sdefinition}
Let $X$ be a threefold with terminal $\QQ$-factorial singularities.
An \emph{extremal blowup} of $X$ is a birational morphism $f: \tilde{X} \to X$ 
such that 
$\tilde{X}$ has only terminal $\QQ$-factorial singularities and $\uprho(\tilde{X}/X)=1$.
\end{sdefinition}
Note that in this situation the divisor $-K_{\tilde X}$ is $f$-ample. 

\begin{stheorem}[{\cite{Kawamata:Div-contr}}]
\label{thm:K:blowup}
Let $X\ni P$ a cyclic quotient terminal threefold singularity of type $\frac1r (a, r-a,1)$, $r\ge 2$
and let $f:\tilde{X} \to X$ be 
an extremal blowup $f:\tilde{X} \to X\ni P$ whose center contains $P$.
Then $f$ is the weighted 
blowup with the system of weights $\frac1r (a, r-a,1)$, in particular, the 
discrepancy of the $f$-exceptional divisor equals $1/r$.
\end{stheorem}

\begin{stheorem}[{\cite{Kawamata-1992-e-app}}]
\label{thm:K:di}
Let $X\ni P$ a terminal threefold singularity of index $r$.
Then there exists an exceptional divisor $E$ over $P$ whose discrepancy equals $1/r$. 
\end{stheorem}

The following useful fact is a consequence of the classification of 
extremal blowups \cite{Kawakita:hi} (see \cite[Lemma~2.6]{P:2013-fano} for explanations).
\begin{slemma}
\label{lemma:discr}
Let $X\ni P$ be a threefold terminal point of index $r>1$
with basket $\B(X,P)$, 
let $f: \tilde X\to X$ be an extremal blowup with $f(E)=P$,
where $E$ is the exceptional divisor, and let $\alpha$ be the discrepancy of $E$.
\begin{enumerate}
\item 
If $X\ni P$ is a point of type other than $\mathrm{cA}/r$ and $r>2$, then 
$\alpha=1/r$.
\item 
If $X\ni P$ is of type $\mathrm{cA}/r$
and $\B(X,P)$ consists of $n$ points of index $r$, then $\alpha=a/r$, where $n\equiv 0\mod a$.
\end{enumerate}
\end{slemma}

\subsection{Du Val del Pezzo surfaces}
Let $S$ be a del Pezzo surface with
only Du Val singularities. We assume that $\uprho(S)=1$.
The definitions of Fano indices and the fundamental divisor 
$A_X$ are applicable to $S$ (see \eqref{eq:deq-q} and \eqref{eq:def-AX}).
Recall also our Convention~\ref{conv:def-AXa}.
Let $\mu: \tilde S\to S$ be the minimal resolution.
We say that a curve $L\subset S$ is a \emph{line} if there exists a $(-1)$-curve $E\subset \tilde S$
such that~$L=\mu (E)$.

\begin{slemma}
\label{lemma:DP:0}
In the above notation assume that $d:=K_S^2<8$.
One has:
\begin{enumerate} 
\item
\label{lemma:DP:0--0} 
The set of lines on $S$ is finite and non-empty. 
\item
\label{lemma:DP:0--1}
The group $\Cl(S)$ is generated by the classes of lines. 
\item
\label{lemma:DP:0--2}
For every effective Weil divisor $D$ on $S$ there is a presentation
\[
D\sim a_0(-K_S)+\sum a_iL_i,
\]
where the $L_i$ are lines in $S$, the $a_i$ are non-negative integers, and 
$a_0=0$ if $d>1$.
\item 
\label{lemma:DP:0-1}
For any line $L$ on $S$ we have $L\qq A_S$, hence $\qQ(S)=d$. 

\item 
\label{lemma:DP:0-2} 
If $D$ is a divisor on $S$ such that $D\qq A_S$, then either
\begin{enumerate} 
\item 
\label{lemma:DP:0a}
$\dim |D|=0$ and $D\sim L$, where $L$ is a line, or
\item 
\label{lemma:DP:0b}
$\dim |D|=1$, $d=1$, and $D\sim -K_S$.
\end{enumerate} 

\item 
\label{lemma:DP:0-3a}
If $D$ is an ample divisor, then $|D|\neq \varnothing$.

\item 
\label{lemma:DP:0-3}
If $D$ is an effective divisor such that $\dim |D|=0$ and $D\neq 0$, then 
$D$ is a line.
 
\item 
\label{lemma:DP:0-6}
Assume that $d>1$.
Then for any two lines $L_1$ and $L_2$ on $S$ the divisor $L_1-L_2$ is a non-trivial torsion element in
$\Cl(S)$. In particular,
$\Cl(S)\simeq \ZZ$ if and only if $S$ contains exactly one line.
\end{enumerate}
\end{slemma}

\begin{proof}[Sketch of the proof]
The assertion~\ref{lemma:DP:0--0} follows from the cone theorem applied to $\tilde S$.
For~\ref{lemma:DP:0--1} and~\ref{lemma:DP:0--2} we refer to \cite[Lemma~2.9]{CP:dP}.
The assertion~\ref{lemma:DP:0-1} follows from the equality
$d=-dK_S\cdot L=\qQ(S) (-K_S\cdot A_S)$.
To prove~\ref{lemma:DP:0-2} assume that $D\not\sim -K_S$ and $D\qq A_S$.
By the orbifold 
Riemann-Roch formula \cite{Reid:YPG} applied to $D$ and $-D$ we have:
\begin{eqnarray*}
\chi(S,\OOO_S(D))&=& \frac12 D\cdot (D-K_S) +1+\sum c_P(D),
\\
\chi(S,\OOO_S(-D))&=& \frac12 D\cdot (D+K_S) +1+\sum c_P(-D).
\end{eqnarray*}
By the Kawamata-Viehweg vanishing and Serre duality $\h^i(S, \OOO_S(- D))=0$ 
for $i=0,1,2$ and $\h^i(S, \OOO_S( D))=0$ for $i=1,2$. Since $c_P(-D)=c_P(D)$,
we obtain
$\h^0 (S,\OOO_S(D)) =1$. Hence we may assume that $D$ is effective.
In this case~\ref{lemma:DP:0-2} is a consequence of~\ref{lemma:DP:0--2}
The assertion~\ref{lemma:DP:0-3a} follows from 
\ref{lemma:DP:0-2}.
For~\ref{lemma:DP:0-3}, in view of~\ref{lemma:DP:0--2} it is sufficient to show 
that $-K_S\cdot D=1$. Let $C\in |-K_S|$ be a general element. Then $C$ is a 
smooth elliptic curve lying in the smooth part of $S$. Since $H^1(S, 
\OOO_S(D+K_S))=0$, from the exact sequence
\[
0\longrightarrow \OOO_S(D+K_S) \longrightarrow \OOO_S (D)
\longrightarrow \OOO_C(D) \longrightarrow 0
\]
 we obtain $\h^0(C, \OOO_C(D))=1$, hence $D\cdot C=\deg \OOO_C(D)=1$.

Finally,~\ref{lemma:DP:0-6}
follows from~\ref{lemma:DP:0--0},~\ref{lemma:DP:0--1} and~\ref{lemma:DP:0-2}.
\end{proof} 

\begin{slemma}
\label{lemma:DP}
Let $S$ be a del Pezzo surface with
only Du Val singularities of type \type{A_n} and $\uprho(S)=1$.
If the group $\Cl(S)$ is torsion free, 
then there are only the following possibilities: 
\begin{table}[H]
{\rm
\setcounter{NN}{0}
\renewcommand{\arraystretch}{1.2}
\begin{tabularx}{0.7\textwidth}{|l|l|X|}
\hline
\multicolumn{1}{|c|}{$K_S^2$} & \multicolumn{1}{c|}{$S$} & \multicolumn{1}{c|}{$\h_S(t)$} 
\\\hline
$9$&$\PP^2$ & $1+3t+6t^2+10t^3+15t^4+21t^5+\cdots $ \\
$8$&$\PP(1^2,2)$ &$1+2t+4t^2+6t^3+9t^4+12t^5+\cdots $ \\
$6$&$\PP(1,2,3)$ & $1+\phantom{2}t+2t^2+3t^3+4t^4+5t^5+\cdots $ \\
$5$&$S_6\subset \PP(1,2,3,5)$ &$1+\phantom{2}t+2t^2+3t^3+4t^4+6t^5+\cdots $
\\ \hline
\end{tabularx}
} 
\end{table} 
\noindent
where $S_6\subset \PP(1,2,3,5)$ is a hypersurface of degree $6$ in $\PP(1,2,3,5)$ having a unique singular point
which 
point of type \type{A_4}; this surface is unique up to isomorphism. 
\end{slemma}

\begin{proof}
The classification can be found in
\cite{Miyanishi-Zhang:88} and computation of $\h_S(t)$ follows from 
the orbifold Riemann-Roch \cite{Reid:YPG}. 
\end{proof} 

\section{$\QQ$-Fano threefolds with torsion in the divisor class group}

In this section we discuss $\QQ$-Fano threefolds whose class group $\Cl(X)$ contains
non-trivial torsion. 

\begin{proposition}[see {\cite[\S~3]{P:2019:rat:Q-Fano}}]
\label{prop:tor}
Let $X$ be a $\QQ$-Fano threefold with $\qQ(X)\ge 5$ and $|\Clt{X}|=N>1$.
Then $\qQ(X)=5$ or $7$ and $X$ belongs to one of the following classes:
\begin{center}
\rm
\setcounter{NN}{0}
\renewcommand{\arraystretch}{1.2}
\begin{tabularx}{0.99\textwidth}{|l|l|l|c|X|}
\hline
&\multicolumn{1}{c|}{$A_X^3$} & \multicolumn{1}{c|}{$\B(X)$} & \multicolumn{1}{c|}{$\g(X)$}
&\multicolumn{1}{c|}{$\h_X(t,\sigma)$ }
\\\hline
\multicolumn{5}{|c|}{\textbf{$\qQ(X)=7$}, $N=2$}
\\\hline
\nr & $1/24$&
$(2^2, 3, 4, 8)$&
$6$&
$1+ t \sigma +t^2 +t^2 \sigma + 2 t^3 + 2 t^3 \sigma + 3 t^4 + 3 t^4 \sigma + 4 t^5 + 4 t^5 \sigma + \cdots$
\\
\nr\label {tab:h:q=7t:d=1/30} &$1/30$&
$(2, 6, 10)$&
$5$&
$1+t +t^2 +t^2 \sigma +t^3 + 2 t^3 \sigma + 2 t^4 + 3 t^4 \sigma + 3 t^5 + 4 t^5 \sigma + \cdots $
\\\hline
\multicolumn{5}{|c|}{\textbf{$\qQ(X)=5$, $N=3$}}
\\\hline
\nr & 
$1/18$& $(2, 9^2)$ & $2$ &
$1+t + t^2 +t^2 \sigma +t^2 \sigma^2 +t^3 + 2 t^3 \sigma + 2 t^3 \sigma^2 + \cdots$
\\\hline
\multicolumn{5}{|c|}{\textbf{$\qQ(X)=5$, $N=2$}}
\\\hline
\nr & 
$1/6$& $(2, 4^2, 6)$ & $10$ & 
$1+t +t \sigma + 2 t^2 + 3 t^2 \sigma + 4 t^3 + 5 t^3 \sigma + 8 t^4 + 7 t^4 \sigma + 12 t^5 + 11 t^5 \sigma + \cdots$
\\\nr & 
$1/8$& $(2^2, 4, 8)$ & $7$ &
$1+t +t \sigma + 2 t^2 + 2 t^2 \sigma + 3 t^3 + 4 t^3 \sigma + 6 t^4 + 6 t^4 \sigma + 9 t^5 + 9 t^5 \sigma + \cdots$
\\\nr & 
$1/12$& $(4^2, 12)$ & $4$ &
$1+t + t^2 +t^2 \sigma + 2 t^3 + 2 t^3 \sigma + 4 t^4 + 4 t^4 \sigma + 6 t^5 + 6 t^5 \sigma + \cdots$
\\\nr & 
$1/28$& $(2, 4, 14)$ & 
$1$ &
$1+t + t^2 + t^3 +t^3 \sigma + 2 t^4 + 2 t^4 \sigma + 3 t^5 + 3 t^5 \sigma +\cdots$ 
\\\hline
\end{tabularx}
\end{center}
In particular, $\p_1(X)=1$.
\end{proposition}

The following fact is a consequence of computer calculations as explained in Appendix~\ref{sect:app}.
In principle, one can perform them by hand but since they are not conceptual it 
is 
more reasonable to use a computer.
\begin{proposition}
\label{prop:qW-neq-qQ}
Let $X$ be a $\QQ$-Fano threefold with $\qQ(X)\ge 3$ and $\qW(X)\neq \qQ(X)$.
Then $\qQ(X)=3$ or $4$ and $X$ belongs to one of the following classes:
\begin{table}[H]
\rm
\setcounter{NN}{0}
\renewcommand{\arraystretch}{1.2}
\begin{tabularx}{0.99\textwidth}{|l|l|l|l|c|X|}
\hline
&\multicolumn{1}{c|}{$A_X^3$} & \multicolumn{1}{c|}{$\B(X)$} & \multicolumn{1}{c|}{$\g(X)$}
& \multicolumn{1}{c|}{$\p_1(X)$} 
&\multicolumn{1}{c|}{$\h_X(t,\sigma)$ }
\\\hline
\multicolumn{6}{|c|}{\textbf{$\qQ(X)=3$}, $\Clt{X}\simeq \ZZ/3\ZZ$}
\\\hline
\nr \label{prop:qW-neq-qQ:1} & $\fracf 12$ & $(3^4, 6)$ & $6$&$2$
& $1+(1+\sigma+2\sigma^2)t+(3+4\sigma+4\sigma^2)t^2+(8+8\sigma+7\sigma^2)t^3+\cdots$
\\
\nr \label{prop:qW-neq-qQ:2} & $\fracf 1{10}$ & $(3^4, 5, 6)$ & $0$&$1$
& $1+\sigma^2t+(\sigma+\sigma^2)t^2+(2+2\sigma+\sigma^2)t^3+\cdots$
\\
\nr \label{prop:qW-neq-qQ:3} & $\fracf 14$ & $(2, 3^2, 12)$ & $2$&$1$
& $1+(\sigma+\sigma^2)t+2(1+\sigma+\sigma^2)t^2+4(1+\sigma+\sigma^2)t^3+\cdots$
\\\hline
\multicolumn{6}{|c|}{\textbf{$\qQ(X)=4$}, $\Clt{X}\simeq \ZZ/2\ZZ$}
\\ \hline
\nr \label{prop:qW-neq-qQ:4} & $\fracf 13$ & $(2^5, 6)$ & $ 10$&$2$
&$1+(1+2\sigma)t+(4+3\sigma)t^2+7(1+\sigma)t^3+12(1+\sigma)t^4+\cdots$
\\
\nr \label{prop:qW-neq-qQ:5} & $\fracf 2{15}$ & $(2^5, 5, 6)$ & $ 3$&$1$
& $1+\sigma t+(2+\sigma)t^2+3(1+\sigma)t^3+5(1+\sigma)t^4+\cdots$
\\
\nr \label{prop:qW-neq-qQ:6} & $\fracf 1{21}$ & $(2^5, 6, 7)$ & $ 0$&$1$
& $1+\sigma t+t^2+(1+\sigma)t^3+2(1+\sigma)t^4+\cdots$
\\
\nr \label{prop:qW-neq-qQ:7} & $\fracf 25$ & $(2^3, 10)$ & $ 12$&$2$
& $1+(1+2\sigma)t+4(1+\sigma)t^2+8(1+\sigma)t^3+14(1+\sigma)t^4+\cdots $
\\
\nr \label{prop:qW-neq-qQ:8} & $\fracf 1{15}$ & $(2^3, 3, 10)$ & $ 1$&$1$
& $1+\sigma t+(1+\sigma)t^2+2(1+\sigma)t^3+3(1+\sigma)t^4+\cdots$
\\
\nr \label{prop:qW-neq-qQ:9} & $\fracf 15$ & $(2^3, 5, 10)$ & $ 5$&$1$
& $1+\sigma t+2(1+\sigma)t^2+4(1+\sigma)t^3+7(1+\sigma)t^4+\cdots $
\\
\nr \label{prop:qW-neq-qQ:10} & $\fracf 4{35}$ & $(2^3, 7, 10)$ & $ 2$&$1$
& $1+\sigma t+(1+\sigma)t^2+2(1+\sigma)t^3+4(1+\sigma)t^4+\cdots$
\\\hline
\end{tabularx}
\end{table}

\end{proposition} 

\begin{proposition}
\label{prop:F-tor4}
Let $X$ be a $\QQ$-Fano threefold with $\qQ(X)\ge 3$, $\qW(X)\neq \qQ(X)$, and $\p_1(X)\ge 2$.
Then one of the following holds
\begin{enumerate} 
\item
\label{prop:F-tor4a}
$X$ is of type \xref{prop:qW-neq-qQ:1} of Proposition~\xref{prop:qW-neq-qQ} and then $X$ is the quotient $X'/\mumu_3$,
where $X'$ is a hypersurface of degree $3$ in $\PP(1^4,2)$:
\[
\left\{x_2x_1^{(1)}+\phi_3\big(x_1^{(1)},\dots, x_1^{(4)}\big)=0\right\}/\mumu_3(0,0,1,1,-1).
\] 
\item
\label{prop:F-tor4b}
$X$ is of type \xref{prop:qW-neq-qQ:4} of Proposition~\xref{prop:qW-neq-qQ} and then $X$ is the quotient $X'/\mumu_2$,
where $X'$ is a hypersurface of degree $4$ in $\PP(1^3,2,3)$:
\[
\left\{x_3x_1+x_2^2+\phi_4(x_1',x_1'')+\phi_2(x_1',x_1'')x_1^2=0\right\}/\mumu_2(0,1,1,1,0).
\]
\end{enumerate}
Here the subscript is the degree of the corresponding variable or polynomial.

In both cases $X$ is rational.
\end{proposition}

\begin{proof}
By Proposition~\ref{prop:qW-neq-qQ} we have either $\qQ(X)=3$ or $\qQ(X)=4$.
First, consider the case $\qQ(X)=3$. Then the variety
$X$ is of type \xref{prop:qW-neq-qQ:1} of Proposition~\xref{prop:qW-neq-qQ}.
The group $\Clt{X}\simeq \ZZ/3\ZZ$ defines a triple cyclic cover $\pi: X'\to X$ which is \'etale outside $\Sing(X)$. Thus $X=X'/\mumu_3$, where
$X'$ is a Fano threefold with terminal singularities (not necessarily $\QQ$-Fano) such that
$\qW(X')$ is divisible by $3$, $A^3_{X'}=3/2$, and $\B(X')=(2)$.
By Theorem~\ref{thm:QFanoF} the variety $X'$ is a hypersurface of degree $3$ in $\PP(1^4,2)$.
In these settings, 
the unique point $P'\in X'$ of index $2$ has coordinates $(0,0,0,0,1)$.
Since $P'\in X'$ is a terminal cyclic quotient singularity, the variable $x_2$ of degree $2$ appears in the equation of $X'$.

The embedding $X'\subset \PP(1^4,2)$ is canonical, 
so it is $\mumu_3$-equivariant and the action of $\mumu_3$ on $X'$ is induced by a liner action on the ambient space $\PP(1^4,2)$.
Moreover, the homogeneous coordinates $x_1^{(k)}$ and $x_2$ can be taken to be semi-invariant.
Modulo a linear coordinate change we may assume that the equation of $X'$ is as follows:
\[
x_2x_1^{(1)}+\phi(x_1^{(1)},\dots, x_1^{(4)})=0.
\]
Consider the affine chart $U_2:=\{x_2\neq 0\}\subset X'$.
Then 
\[
U_2=\{y^{(1)}+\phi'(y^{(1)},\dots, y^{(4)})=0\}/\mumu_2(1,1,1,1).
\]
The quotient $U_2/\mumu_3$ is a terminal cyclic quotient singularity.
Hence $\mumu_3$-action on $U_2$ has weights $(a,1,1,-1)$ for some $a\in \{0,1,2\}$ modulo permutations 
of $y^{(2)},y^{(3)}, y^{(4)}$ and changing the generator of $\mumu_3$.
Thus $\mumu_3$-action on $\PP(1^4,2)$ has weights $(0,a,1,1,-1)$.
If $a=1$, then the fixed point locus contains the curve $\{x_2=x_1^{(4)}=0\}\cap X'$, a contradiction. 
If $a=-1$, then, by the same reason, the line $x_2=x_1^{(2)}=x_1^{(3)}=0$ is not contained in 
$X'$ and so $\phi(x_1^{(1)},0,0, x_1^{(4)})\neq 0$. But in this case $\phi$ and $x_2x_1^{(1)}$ must be $\mumu_3$-invariant, a contradiction.
Hence $a=0$. This proves~\ref{prop:F-tor4a}.

Now consider the case $\qQ(X)=4$. Then again by Proposition~\ref{prop:qW-neq-qQ} the variety
$X$ is of type \xref{prop:qW-neq-qQ:4} or \xref{prop:qW-neq-qQ:7}.
The group $\Clt{X}\simeq \ZZ/2\ZZ$ defines a double cover $\pi: X'\to X$ which is \'etale outside $\Sing(X)$, where
$X'$ is a Fano threefold with terminal singularities. Moreover,
in our two cases we have
\begin{itemize}
\item[{\ref{prop:qW-neq-qQ:4}}.]
$\qW(X')=4$, $A^3_{X'}=2/3$, $\B(X')=(3)$, $\dim |A_{X'}|=2$, $\g(X')=22$;
\item[{\ref{prop:qW-neq-qQ:7}}.]
$\qW(X')=4$, $A^3_{X'}=4/5$, $\B(X')=(5)$, $\dim |A_{X'}|=2$, $\g(X')=26$.
\end{itemize}
By Theorem~\ref{thm:QFanoF} in the case
\xref{prop:qW-neq-qQ:4}
the variety $X'$ is a hypersurface of degree $4$ in $\PP(1^3,2,3)$ and 
by Proposition~\ref{prop:wF4} below the case \xref{prop:qW-neq-qQ:7} does not occur.
As above
the embedding $X'\subset \PP(1^3,2,3)$ is $\mumu_2$-equivariant, where 
the action on $X'$ is induced by a liner action on the ambient space $\PP(1^3,2,3)$.
Moreover, we may assume that the coordinates $x_1,x_1',x_1'',x_2,x_3$ are semi-invariants
with eigenvalues $\pm 1$ and $x_1$ is an invariant. 
Since $\B(X')=(3)$, the terms $x_2^2$ and some of $x_3x_1$, $x_3x_1'$ or $x_3x_1''$
appear in the equation. Modulo an obvious coordinate change we may assume that the equation of $X'$ is as follows:
\[ 
x_3x_1+x_2^2+\phi(x_1,x_1',x_1'')=0,
\]
where $\phi$ is a semi-invariant of degree $4$. Then we see that $\phi$ is in fact an invariant and $x_3$ must be an invariant as well. Since the set of fixed points is finite, the variables $x_1',x_1'', x_2$ cannot be invariant. Hence the action have the desired form.
The rest is obvious.
\end{proof} 

\begin{sproposition}
\label{prop:wF4}
Let $Y$ be a weak Fano threefold with terminal singularities with $\g(Y)> 22$.
Assume that there exists a Weil divisor $B$ on $Y$ such that $-K_Y\sim 4B$ and $\dim |B|\ge 2$.
Then $Y\simeq\PP^3$. 
\end{sproposition}

\begin{proof}
If $Y$ is a $\QQ$-Fano threefold, then the assertion follows from \cite[Theorem~1.2]{P:2013-fano}.
Thus we assume that $Y$ is not $\QQ$-Fano.
Replacing $Y$ with its $\QQ$-factorialization, we may assume that $Y$ is $\QQ$-factorial.
Then $\uprho(Y)>1$.
Run the MMP on $Y$:
\begin{equation}
\label{eq:MMP-Y}
Y=Y^{(0)} \dashrightarrow Y^{(1)} \dashrightarrow \cdots \dashrightarrow Y^{(n-1)} \dashrightarrow Y^{(n)}.
\end{equation} 
Let $\BBB_k$ be the proper transform of the linear system $\BBB_0:=|B|$ on $Y^{(k)}$. 
On each step the relation $-K_{Y^{(k)}}\sim 4 \BBB_k$ is kept.
Moreover, $\dim \BBB_k=\dim |B|\ge 2$. Thus $Y^{(n)}$ has a structure of Mori-Fano fiber space $\varphi: Y^{(n)}\to Z$ such that
for a general fiber $F$ we have $-K_{F}\sim 4 \BBB_n|_F$. This is not possible if $\varphi$ is a del Pezzo or
a rational curve fibration. Hence $Z$ is a point and $Y^{(n)}$ is a $\QQ$-Fano threefold such that 
$\qW(Y^{(n)})$ is divisible by $4$ and $\g(Y^{(n)})\ge \g(Y)> 22$. 
By \cite[Theorem~1.2]{P:2013-fano} we have $Y^{(n)}\simeq \PP^3$.
Let us consider the last step $\psi: Y^{(n-1)} \dashrightarrow Y^{(n)}$ of the MMP.
Since $\uprho(Y^{(n)})=1$, the map $\psi: Y^{(n-1)} \dashrightarrow Y^{(n)}$ must be a divisorial contraction and
since $-K_{Y^{(n-1)}}$ is divisible, $\psi$ cannot be a contraction of a divisor to a curve.
Thus $\psi$ contracts a divisor $E\subset Y^{(n-1)}$ to a point $P\in Y^{(n)}$. 
Since $-K_{Y^{(n-1)}}\sim 4\BBB_{n-1}$ is $\psi$-ample, we see that $P\in \Bs \BBB_{n}$.
Then $\BBB_n$ is subsystem in $|\OOO_{\PP^3}(1)|$ of codimension $1$ and
so $\BBB_n$ has a single base point, say $P$.
Therefore, $\BBB_{n-1}$ has no base points outside $E$. Since $-K_{Y^{(n-1)}}$ is ample on $E$, we see that
$-K_{Y^{(n-1)}}$ is nef. Since $-K_{Y^{(n-1)}}$ is the proper transform of $-K_Y$, it is also big.
Thus $Y^{(n-1)}$ is a weak Fano threefold with terminal singularities and $-K_{Y^{(n-1)}}\sim 4\BBB_{n-1}$. 
According to \cite{Kawakita:smooth} the contraction $\psi$ is the weighted blowup with weights $(1,w_1,w_2)$,
where $\gcd(w_1,w_2)=1$. Then we have 
\begin{equation*}
\B(Y^{(n-1)})=(w_1,w_2),\qquad K_{Y^{(n-1)}}=\psi^*K_{\PP^3}+(w_1+w_2)E.
\end{equation*} 
Here $w_1+w_2$ is divisible by $4$ because so $K_{Y^{(n-1)}}$ is. 
Then
\begin{equation*}
0<(-K_{Y^{(n-1)}})^3=64-\frac{(w_1+w_2)^3}{w_1w_2}.
\end{equation*} 
Up to permutation of $w_1$ and $w_2$ we obtain the following possibilities
for $\left(w_1,w_2, A_{Y^{(n-1)}}^3\right)$:
\[
(1,3, 2/3),\qquad (3,5,7/15),\qquad (5,7,8/35).
\]
Then from the Kawamata-Viehweg vanishing and the orbifold Riemann-Roch formula (see~\eqref{eq:RR:g}) we obtain 
$\g(Y^{(n-1)}) =22, 15, 7$
in these cases, respectively. 
This contradicts our assumption $\g(Y)>22$.
\end{proof}

\begin{scorollary}
\label{cor:q6-7}
Let $X$ be a $\QQ$-Fano threefold with $\qQ(X)\ge6$.
If either $\p_1(X)\ge 2$ or $\qQ(X)\ge 7$ and $\p_2(X)\ge 2$, then $X$ is rational.
\end{scorollary}

\begin{proof}
By Theorem~\ref{thm:q8} we may assume that $\qQ(X)=6$ or $7$.
By Proposition~\ref{prop:tor} 
the group $\Cl(X)$ is torsion free.
If $\qQ(X)=6$ and $\p_1(X)\ge 2$, then computer search (see Sect.~\ref{sect:app}) gives only one possibility: 
$A_X^3=1/5$, $\B(X)=(5)$, \cite[\# 41469]{GRD}.
In this case $X$ is rational by Theorem~\ref{thm:QFanoF}.
Similarly in the case $\qQ(X)=7$ and $\p_2(X)\ge 2$ there are four possibilities:
\begin{itemize} 
\item 
$A_X^3= 1/6$, $\B(X)=(2, 3)$, $\g(X)= 29$, \cite[\# 41492]{GRD}; 
\item
$A_X^3= 1/12$, $\B(X)=(2, 3^2, 4)$, $\g(X)= 14$, \cite[\# 41484]{GRD}; 
\item
$A_X^3= 1/10$, $\B(X)=(2^3, 5)$, $\g(X)= 17$, \cite[\# 41489]{GRD}; 
\item
$A_X^3= 1/15$, $\B(X)=(2^2, 3, 5)$, $\g(X)= 11$, \cite[\# 41481]{GRD}.
\end{itemize}
In the case $A_X^3= 1/6$ we have $X\simeq \PP(1^2,2^2)$ by Theorem~\ref{thm:QFanoF} and 
in the cases $A_X^3= 1/12$ and $A_X^3= 1/10$ there is explicit descriptions of $X$ as a weighted hypersurface
$X_6\subset \PP(1,2,3^2,4)$ and $X_6\subset \PP(1,2^2,3,5)$, respectively (see \cite[Theorem~1.4]{P:2013-fano}
and \cite[Theorem~1.2]{P:2016:QFano7}). It is easy to see that these varieties are rational. 
Rationality of $X$ in the case $A_X^3= 1/15$ was proved in \cite[Proposition~5.1]{P:2019:rat:Q-Fano}.
\end{proof}

\begin{proposition}
\label{prop:tor:new}
Let $X$ be a $\QQ$-Fano threefold with $\qQ(X)\ge 3$ and $N:=|\Clt{X}|>1$. 
Assume either $N\ge 4$
or $\qQ(X)=4$ and $N=3$.
Then $\Clt{X}\simeq \ZZ/N\ZZ$ and $X$ is the quotient $X'/\mumu_N$, where $X'$ is a Fano threefold with terminal singularities
and the action of $\mumu_N$ on $X'$ is free outside a finite number of points.
Moreover, $X$ and $X'$ are described by the following table:
\begin{table}[H]
\rm
\setcounter{NN}{0}
\renewcommand{\arraystretch}{1.2}
\begin{tabularx}{0.95\textwidth}{|l|l|l|l|l|l|l|l|X|}
\hline 
& \multicolumn{1}{c|}{$\qQ(X)$} & \multicolumn{1}{c|}{$\B(X)$} & \multicolumn{1}{c|}{$N$} & \multicolumn{1}{c|}{$A_X^3$} & \multicolumn{1}{c|}{$\g(X)$} & \multicolumn{1}{c|}{$\p_1(X)$} & $\B(X')$ & \multicolumn{1}{c|}{$X'$} 
\\\hline
\nr\label{prop:tor:newA} & $3$ & $(2^2,8^2)$ & $4$ & $1/4$ & $2$ & $1$ & $(2^2)$ & $X'_4\subset \PP(1^3,2^2)$ 
\\
\nr\label{prop:tor:newC} & $3$ & $(5^4)$ & $5$ & $2/5$ & $4$ & $1$ & $\varnothing$ & $Q\subset \PP^4$ 
\\
\nr\label{prop:tor:newB} & $4$ & $(5^4)$ & $5$ & $1/5$ & $5$ & $1$ & $\varnothing$ & $\PP^3$
\\
\nr\label{prop:tor:newD} & $4$ & $(9^2)$ & $3$ & $1/9$ &$3$& $1$& $(3^2)$ & $X'_6\subset \PP(1^2,2, 3^2)$
\\\hline
\end{tabularx}
\end{table}
\noindent
where $X'_4\subset \PP(1^3,2^2)$ is a hypersurface of degree $4$
and $Q\subset \PP^4$ is a smooth quadric.

In particular, $X$ is rational.
\end{proposition}

\begin{proof}
The computer search by the algorithm outlined in Appendix~\ref{sect:app}
produces exactly three possibilities with numerical invariants as in the table.
Then $\Clt{X}\simeq \ZZ/N\ZZ$, where $N=4$ or $5$ and, as in the proof of Proposition~\ref{prop:F-tor4},
we see that the generator of this group
defines a global cyclic cover $X'\to X$ of degree $N$. Thus $X=X'/\mumu_N$.
Here the Gorenstein index of $X'$ is strictly less than $\qQ(X')$. 
Hence $X'$ by Theorem~\ref{thm:QFanoF}\ $X'$ is either a hypersurface of degree $4$ in $\PP(1^3,2^2)$,
of degree $6$ in $\PP(1^2,2, 3^2)$, the projective space 
$\PP^3$, or a quadric $Q\subset \PP^4$.

It remains to show that $X$ is rational. In the case~\ref{prop:tor:newB}
the variety $X$ is toric, so rationality is obvious. 
In the case~\ref{prop:tor:newC} the action of $\mumu_5$ on $Q$ is induced by a 
linear action on $\PP^4$. The projection $Q \dashrightarrow \PP^3$ from a fixed point $P\in Q$
is equivariant, hence $Q/\mumu_5$ is birational to the toric variety $\PP^3/\mumu_5$,
so it is rational.

Consider the case~\ref{prop:tor:newA}.
As above, action of $\mumu_4$ on $X'$ is induced by a liner action on the ambient space $\PP(1^3,2^2)$.
Moreover, we may assume that the coordinates $x_1,x_2,x_3,y_1,y_2$ are semi-invariants with
$\deg x_i=1$, $\deg y_j=2$. The equation of $X'$ can be written as follows
\[
q(y_1,y_2)+y_1\phi_2'(x_1,x_2,x_3)+y_2\phi_2''(x_1,x_2,x_3)+\phi_4(x_1,x_2,x_3) =0,
\]
where $q$, $\phi_2'$, $\phi_2''$, $\phi_4$ are homogeneous polynomials of degree $2$, $2$, $2$, $4$, 
respectively.
Since the line $\Sing(\PP(1^3,2^2))=\{x_1=x_2=x_3=0\}$ is not contained in $X'$, we have $q(y_1,y_2)\neq 0$. 
The projection $\Psi: X' \dashrightarrow \PP(1^3)=\PP^2$
is an equivariant rational map whose fibers are conics in $\PP(1,2^2)=\PP^2$
and whose indeterminacy locus $\operatorname{Ind}(\Psi)=\{q(y_1,y_2)=x_1=x_2=x_3=0\}$
consists of one or two points of index $2$.
These points cannot be switched by the action of $\mumu_4$ because their 
images on $X$ are exactly the points of index $8$.
Hence points in $\operatorname{Ind}(\Psi)=\{q(y_1,y_2)=x_1=x_2=x_3=0\}$ give invariant sections of $\Psi$.
This implies that the rational curve fibration $X=X'/\mumu_4 \dashrightarrow \PP^2/\mumu_4$
has a section, hence $X$ is rational.

Finally, in the case~\ref{prop:tor:newD}, as above, the equation of $X'$ can be reduced to one of the following forms
\begin{eqnarray*}
x_3x_3'&+&a x_2^3 + x_2\phi_4(x_1,x_1')+\phi_6(x_1,x_1') =0,
\\
x_3^2&+&a x_2^3 + x_2\phi_4(x_1,x_1')+\phi_6(x_1,x_1') =0.
\end{eqnarray*}
In the former case the projection to $\PP(1^2,2,3)$ establishes the rationality of $X$.
In the latter case the point $(0,0,0,0,1)$ on $X'$ is a unique non-Gorenstein point
and it is not of type \type{cA/3} \cite{Reid:YPG}. Hence its quotient by $\mumu_3$ cannot be terminal, a contradiction.
\end{proof} 
 
\begin{proposition}
\label{prop:tq5}
Let $X$ be a $\QQ$-Fano threefold with $\Clt{X}\neq 0$.
Assume either 
\begin{enumerate}
\item 
$\qQ(X)=5$, and $\p_2(X)\ge 2$, or
\item
$\qQ(X)=\qW(X)= 3$, $\p_1(X)\ge 2$, and $\dim |A_X|\le 0$.
\end{enumerate}
Then $\Clt{X}\simeq \ZZ/2\ZZ$ and $X$ is the quotient $X'/\mumu_2$, where $X'$ is a Fano threefold with terminal singularities 
and the action of $\mumu_2$ on $X'$ is free outside a finite number of points.
Moreover, $X$ and $X'$ are described by the following table:
\begin{table}[H]
\rm
\setcounter{NN}{0}
\renewcommand{\arraystretch}{1.2}
\begin{tabularx}{0.95\textwidth}{|l|l|l|l|l|l|l|X|}
\hline 
&\multicolumn{1}{c|}{$\qQ(X)$} &\multicolumn{1}{c|}{$\B(X)$} & \multicolumn{1}{c|}{$A_X^3$} & \multicolumn{1}{c|}{$\g(X)$} 
& $\B(X')$ &\multicolumn{1}{c|}{$\g(X')$}& \multicolumn{1}{c|}{$X'$} 
\\\hline
\nr \label{prop:tq5b} &$5$& $(2^2,4,8)$ & $1/8$ & $7$ 
& $(2,4)$ &$16$& $X_6' \subset \PP(1^2,2,3,4)$
\\
\nr \label{prop:tq5a} & $5$& $(2, 4^2,6)$ & $1/6$ & $10$ 
& $(2^2,3)$ &$21$ &$X_4'\subset \PP(1^2,2^2,3)$ 
\\
\nr \label{prop:tq3c} &$3$& $(2^4, 4^2)$&$1/2$&$6$&$( 2^2)$&$14$&$X_4'\subset \PP(1^3,2^2)$
\\\hline
\end{tabularx}
\end{table}
\noindent
In particular, $X$ is rational.
\end{proposition}

\begin{proof}
The numerical invariants of $X$ and its global cover $X'$ are obtained from Proposition~\ref{prop:tor}
in the case $\qQ(X)=5$
and by computer search in the case $\qQ(X)=3$.
Then in the case~\ref{prop:tq5b} we see that $X'$ is a hypersurface $X_6' \subset \PP(1^2,2,3,4)$
by Theorem~\ref{thm:QFanoF}. 
As above, we may assume that the coordinates $x_1,x_1',x_2,x_3,x_4$ are semi-invariants with
$\deg x_i=i$, $\deg x_1'=1$. 
The set of non-Gorenstein points of $X'$ consists of either 
two cyclic quotient singularities $P_1$ and $P_2$ of types $\frac12(1,1,1)$ and $\frac 14(1,1-1)$ or 
one point $P$ of type $\mathrm{cAx/4}$. The latter case does not occur because the quotient $(X'\ni P)/\mumu_2$
must be a terminal singularity. Hence the equation of $X'$ must contain the term $x_4x_2$.
Then the projection $X'\to \PP(1^2,2,3)$ is equivariant and birational. 
Therefore the variety $X=X'/\mumu_2$ is birational to $\PP(1^2,2,3)/\mumu_2$, so it is rational.

Consider the case~\ref{prop:tq5a}. We claim that $X'$ is a $\QQ$-Fano threefold.
Indeed, otherwise, as in the proof of Proposition~\ref{prop:wF4} we can take $\QQ$-factorialization $Y\to X'$
and run the MMP \eqref{eq:MMP-Y}. 
At the end we obtain a $\QQ$-Fano threefold $Y^{(n)}$
such that $\g(Y^{(n)})\ge \g(X)\ge 21$. Hence $Y^{(n)}\simeq \PP(1^3,2)$ or $Y^{(n)}\simeq Y_4\subset \PP(1^2,2^2,3)$ 
by \cite{P:2013-fano}. In both cases the group $\Cl(Y^{(n)})$ is torsion free.
Therefore, the same is true for $\Cl(Y^{(k)})$, $0\le k\le n$.
Consider the first step $\psi: Y \dashrightarrow Y^{(1)}$. 
Assume that $\psi$ is a flip and let $C$ be a component of the flipping locus.
Then $-K_Y\cdot C=5A_Y\cdot C\ge 5 /\operatorname{Index}(Y) =5/6$.
On the other hand, by \cite[(2.3.2)]{Mori:flip} we have $-K_Y\cdot C =1-\sum_P w_{P}(0)$, where 
the sum runs through the set of points lying on $C$ and $w_{P}$ is a local invariant defined in \cite[(2.2.1)]{Mori:flip}.
By definition $w_{P}$ takes values in $\frac1r \ZZ_{>0}$, where $r$ is the index of $P$.
Since $r\in \{2,\, 3\}$ in our case we obtain $-K_Y\cdot C\le 2/3$, a contradiction.
Therefore, $\psi$ is a divisorial contraction. Let $S$ be the exceptional divisor.
Since $K_Y$ is divisible, $\psi(S)$ is not a curve. Thus $Q:=\psi(S)$ is a point. Let $m$ be its index.
Assume that $m>1$.
Since $-K_{Y^{(1)}}\sim 5 \psi_* A_Y$ and $-K_{Y^{(1)}}$ is a (local) generator of the group $\Clt{Y^{(1)}, Q}\simeq \ZZ/m\ZZ$, 
the numbers $m$ and $5$ must be coprime. 
Write $mK_Y\sim \psi^* (mK_{Y^{(1)}})+m\alpha S$, where $\alpha\in \frac1m \ZZ_{>0}$ is the discrepancy of $S$.
Since the group $\Cl(Y^{(1)})$ is torsion free, this implies that $m\alpha$ is divisible by $5$. 
In particular, $\alpha>1/m$. On the other hand, by Theorem~\ref{thm:K:di} there exists an exceptional over $Q\in Y^{(1)}$
divisor $S'$ whose discrepancy equals $1/m$.
Then the discrepancy of $S'$ over $Y$ must be strictly less than $1/m$.
Since $\B(Y)=(2^2,3)$, the only possibility is $m=2$. 
In this situation $-K_{Y^{(1)}}$ is nef and big, and does not contract any divisors.
Therefore, $Y^{(1)}$ is an almost Fano threefold with terminal singularities
of Gorenstein index $\le 2$. 
The anticanonical model $Y^{(1)}_{\mathrm{can}}$ of $Y^{(1)}$
a Fano threefold whose singularities are also of index $\le 2$.
By Theorem~\ref{thm:QFanoF} we have $Y^{(1)}\simeq \PP(1^3,2)$ because 
$\g(Y^{(1)}_{\mathrm{can}})\ge \g(X')\ge 21$ and 
$\PP(1^3,2)$ is $\QQ$-factorial.
But in this case $Y^{(1)}$ has a unique singular point which is of type $\frac12(1,1,1)$ 
and its extraction $\psi$ produces a smooth variety, a contradiction. 
Therefore, $m=1$, i.e. $Q=\psi(S)$ is a Gorenstein point. Then we can apply the above arguments replacing $Y$ with 
$Y^{(1)}$ and get a contradiction again.
Thus $X'$ is a $\QQ$-Fano threefold. Then
$X'$ is a hypersurface of degree $4$ in $\PP(1^2,2^2,3)$ 
by \cite{P:2013-fano}.
As in the case~\ref{prop:tq5b} we see that the projection $X'\to \PP(1^2,2^2)$ is $\mumu_2$-equivariant and birational. 
Hence $X=X'/\mumu_2$ is rational.

Finally in the case~\ref{prop:tq3c} we see that $X'$ is a hypersurface $X_4'\subset \PP(1^3,2^2)$
by Theorem~\ref{thm:QFanoF}.
As above, we may assume that the coordinates $x_1,x_1',x_1'',x_2,x_2'$ are semi-invariants
with eigenvalues $\pm 1$ and $x_1$ is an invariant. Since the fixed point locus if zero-dimensional, 
the $\mumu_2$-action on $\PP(1^3,2^2)$ has weights $(0,0,1;1,1)$. 
Then $X=X'/\mumu_2$ is rational by the arguments similar to that in the case~\ref{prop:tq5b}.
\end{proof}

\section{Sarkisov link}

The following construction will be systematically used throughout the paper.
From now on we adopt the following notation.
\label{sect:sl}
\subsection{}
\label{subsect:SL}
Let $X$ be a non-Gorenstein $\QQ$-Fano threefold of $\QQ$-Fano index $q=\qQ(X)>1$.
Let $\MMM$ be a linear system on $X$ such that $\MMM\qq nA_X$ with $n<q$,
$\dim \MMM>0$, and $\MMM$ has no fixed components.
This $\MMM$ will be chosen at the beginning and fixed throughout this section.
We usually take $\MMM=|nA_X|$ if $\qQ(X)=\qW(X)$.
Let $c:=\operatorname {ct} (X, \MMM)$ be the canonical threshold of the pair $(X, \MMM)$.
We assume that $c\le 1$ (see Lemma~\ref{lemma:ct} below).
According to {\cite[Proposition~2.10]{Corti95:Sark}} (see also {\cite[Claim~4.5.1]{P:G-MMP}})
there exists an extremal blowup $f: \tilde{X} \to X$.
that is crepant with respect to $K_X+c \MMM$.
By our construction $\uprho (\tilde{X})=2$ and $-(K_{\tilde X}+c \tilde \MMM)$ is nef and big.
As in~\cite{Alexeev:ge}, run the log minimal model program on $\tilde{X}$ with respect to $K_{\tilde{X}}+c \tilde\MMM$
(see e.g. \cite[4.2]{Alexeev:ge} or \cite[12.2.1]{P:G-MMP}).
We obtain the following Sarkisov link:
\begin{equation}
\label{diagram-main}
\vcenter{
\xymatrix@C=19pt{
&\tilde{X}\ar@{-->}[rr]^{\chi}\ar[dl]_{f} && \bar{X}\ar[dr]^{\bar f}
\\
X &&&&\hat{X}
} 
}
\end{equation}
where $\chi$ is an isomorphism in codimension $1$,
the variety $ \bar{X}$ also has only terminal $\QQ$-factorial singularities,
$\uprho (\bar{X})=2$, and
$\bar{f}: \bar{X}\to \hat{X}$ is an extremal $K_{\bar{X}}$-negative Mori contraction
which can be either divisorial or fiber type.

\begin{sremark}
\label{rem:SL:choice}
The proof of \cite[Proposition~2.10]{Corti95:Sark} shows that for any zero-dimensional
canonical center $P$ of the pair $(X, c\MMM)$ there exists an extremal blowup $f: \tilde{X} \to X$
as in~\ref{subsect:SL} with center $P$.
Thus in general the link~\eqref{diagram-main} is not determined 
by the choice of $\MMM$. 
\end{sremark}

\subsection{}
In what follows, for a divisor (or a linear system) $D$ on $X$
by $\tilde D$ and $\bar D$ we denote
proper transforms of $D$
on $\tilde{X}$ and $\bar{X}$, respectively.
By $E$ we denote the $f$-exceptional divisor.
For $1\le k<q$, let $\MMM_k$ be a linear system such that $\MMM_k\qq kA_X$.
As in the case $k=n$, we usually take $\MMM_k=|kA_X|$ if $\qQ(X)=\qW(X)$.
By $M_k$ we denote a general member of $\MMM_k$.
We can write
\begin{equation} 
\label{equation-1}
\begin{array}{lll}
K_{\tilde{X}} &\qq & f^*K_X+\alpha E,\qquad \alpha\in \QQ,\ \alpha>0,
\\[2pt]
\tilde\MMM_k&\qq& f^*\MMM_k- \beta_k E,\qquad \beta_k \in \QQ,\ \beta_k \ge 0.
\end{array}
\end{equation}
Then by taking proper transforms on $\bar X$ we obtain
\begin{equation*}
k K_{\bar{X}}+q \bar\MMM_k \qq (k \alpha -q \beta_k) \bar E.
\end{equation*}
Moreover, if $k K_{X}+q \MMM_k\sim 0$ near $f(E)$, 
then $k \alpha -q \beta_k$ is an integer and we have linear equivalence
\begin{equation}
\label{eq:12-divisors}
k K_{\bar{X}}+q \bar\MMM_k \sim (k \alpha -q \beta_k) E.
\end{equation}
In particular, this holds if $\qQ(X)=\qW(X)$.

\begin{slemma}[{\cite[Lemma~4.2]{P:2010:QFano}}]
\label{lemma:ct}
Let $P\in X$ be a point of index $r>1$. Assume that in a neighborhood of $P$ we have
$\MMM\sim -mK_X$, where $0<m<r$. Then
$\ct(X,\MMM)\le 1/m$. 
Therefore,
\[ 
\beta_n\ge m\alpha\quad\text{and}\quad q\beta_n-n\alpha\ge \alpha>0.
\]
\end{slemma}

\subsection{}
\label{case-bir}
Assume that the contraction $\bar{f}$ is birational. Then $\hat{X}$ is a $\QQ $-Fano threefold.
In this case, denote by $\bar{F}$ the
$\bar{f}$-exceptional divisor, by
$\tilde F \subset \tilde{X}$ its proper transform, $F:=f(\tilde F)$, and
$\hat{q}:=\qQ (\hat{X})$.
The divisor $\bar E$ is not contracted by $\bar f$, i.e.~$\bar{E}\neq \bar{F}$ 
(see e.g. \cite[Claim~4.6]{P:2010:QFano}).
Let $A_{\hat{X}}$ be a fundamental divisor on $\hat{X}$.
Write
\[
F\qq d A_X,\qquad \hat{E} \qq eA_{\hat{X}},\qquad \hat\MMM_k \qq s_kA_{\hat{X}},
\]
where $d,\, e\in \ZZ_{>0}$,\ $s_k \in \ZZ_{\ge0}$. 

Note that $s_k=0$ if and only if 
$\dim \MMM_k=0$ and the unique element $M_k\in \MMM_k$ 
coincides with the $\bar f$ -exceptional divisor $\bar{F}$.

\begin{slemma}
\label{lemma:sl:torsion} 
If $\Clt{X}=0$, then $\Clt{\hat{X}}\simeq \ZZ/\ZZ_{d/e}$. 
\end{slemma}

\begin{proof}
Follows from obvious isomorphisms
\[
\ZZ/d\ZZ\simeq \Cl(X)/(F\cdot \ZZ) \simeq \Cl( \bar{X})/(\bar{F}\cdot \ZZ\oplus\bar{E}\cdot\ZZ)
\simeq \Cl(\hat{X})/(\hat{E}\cdot\ZZ)
\]
and $\Cl(\hat{X})/(\Clt{\hat{X}}\oplus \hat{E}\cdot\ZZ) \simeq \ZZ/e\ZZ$.
\end{proof}

\subsection{}
\label{case-nonbir}
Assume that $\bar{f}$ is a fibration. Then we denote by
$\bar{F}$ a general geometric fiber.
Then $\bar{F}$ is either a smooth rational curve or a del Pezzo surface contained in the smooth part of $\bar X$.
The image of the restriction map $\Cl(\bar{X}) \to \Pic (\bar{F})$ is isomorphic to $\ZZ$.
Let $\Xi$ be its ample generator.
As above, we can write
\[
-K_{\bar{X}} |_{\bar{F}}=-K_{\bar{F}} \sim \hat{q}\Xi,\qquad \bar{E} |_{\bar{F}} \sim e \Xi,\qquad \bar\MMM_k |_{\bar{F}} \sim s_k \Xi,
\]
where $\hat{q},\, e\in \ZZ_{>0}$, and $s_k \in \ZZ_{\ge0}$.

\begin{slemma}
\label{lemma:base-surface}
Assume that $\dim \hat X=2$. Then 
$\hat X$ is a del Pezzo surface with Du Val singularities of type \type{A}, $\uprho(\hat X)=1$, and 
$\Clt{\hat X}\simeq \Clt{\bar X}$.
Furthermore, there is an embedding 
\[
\Clt{\hat X}\subset \Clt{X}. 
\]
In particular, if $\Cl(X)$ is torsion free, then so $\Cl(\hat X)$ is and so $\hat X$ is one of the 
four surfaces described in Lemma~\xref{lemma:DP}.
\end{slemma}

\begin{proof}
By \cite[Theorem~1.2.7]{MP:cb1} the surface $\hat X$ has only 
Du Val singularities of type \type{A_n}.
Since $\uprho(\hat X)=\uprho(\bar X)-1=1$ and $\hat X$ is uniruled, $-K_{\hat X}$ is ample.
Further, since both~$\bar X$ and~$\hat X$ 
have only isolated singularities and $\Pic(\bar X/\hat X)\simeq\mathbb Z$,
there is a well-defined injective map 
\[
\bar f^*\colon \Cl(\hat X) \longrightarrow \Cl(\bar X).
\]
Hence $\Clt{\hat X}\simeq \Clt{\bar X}\simeq \Clt{\tilde X}$.
On the other hand, the push-forward map $f_*: \Cl(\tilde X)\to \Cl(X)$ is 
the quotient by the subgroup $\ZZ\cdot E$, hence $f_*$ is injective on $\Clt{\tilde X}$.
\end{proof}

Regardless of whether $\bar{f}$ is birational or not,
from \eqref{eq:12-divisors} we obtain
\begin{corollary}
\label{cor-eq:main}
In the notation of \xref{case-bir} and \xref {case-nonbir} one has
\begin{equation} 
\label{eq:main}
k \hat{q}=q s_k+(q \beta_k-k \alpha) e,
\end{equation}
where $q \beta_n-n \alpha>0$. If furthermore $\qQ(X)=\qW(X)$, then $q \beta_n-n\alpha$ is a positive integer.
\end{corollary}

Also from our construction we obtain the following easy corollary.
It shows that most of $\QQ$-Fano threefolds of $\QQ$-Fano index at least two 
are not birationally rigid.
In the case where $X$ is a weighted hypersurface much stronger result was proved in \cite{Abban-Cheltsov-Park:Hyp}.
\begin{corollary}
Let $X$ be a $\QQ$-Fano threefold with $\qQ(X)>1$.
If $\p_n(X)\ge 2$ for some $n<\qQ(X)$, then $X$ is not birationally rigid.
\end{corollary}

\begin{proof}
Assume the converse. 
Let $n$ be the minimal positive integer such that $\p_n(X)\ge 2$.
Thus there is a (complete) linear system $\MMM$ such that $\dim \MMM>0$ 
and $\MMM\qq nA_X$. Apply the construction \eqref{diagram-main}.
Since $X$ is birationally rigid, $\hat X\simeq X$ and so $\hat q=q$. 
By \eqref{eq:main} we have $n>s_n$. But then $\dim \hat \MMM>0$ and 
$\hat \MMM\qq s_n A_{\hat X}$. This contradicts our assumption on minimality of $n$.
\end{proof} 

\section{Sarkisov link and rationality}
\begin{lemma}
\label{lemma:q=1}
In the notation of~\xref{subsect:SL} assume that $X$ is not rational.
\begin{enumerate} 
\item \label{lemma:q=1a} 
If the contraction $\bar{f}$ is birational, then either $\hat q\le 6$ or $\hat q=7$ and $s_n\ge 2$. 
\item \label{lemma:q=1b} 
If $\bar{f}$ is a fibration, then $\hat q=1$.
\end{enumerate}
\end{lemma}

\begin{proof}
The assertion~\ref{lemma:q=1a} follows from Theorem~\ref{thm:q8} and Corollary~\ref{cor:q6-7}. 
To prove~\ref{lemma:q=1b} assume that $\hat q>1$.
If $\hat X$ is a curve, then $\hat X\simeq \PP^1$ and $\bar{F}$ is a smooth del Pezzo surface
with divisible canonical class. Thus $\bar f$ is either a generically $\PP^2$-bundle 
or quadric bundle. Then $\bar f$ must be locally trivial in Zariski topology, hence 
$\bar X$ is rational in this case. Now assume that $\dim \hat X=2$.
Let $L$ be an effective Weil divisor on $\hat X$ such that $L\sim A_{\hat X}$ 
and let 
$B$ be a Weil divisor on $\bar X$ whose image is $\Xi$ (see \ref{case-nonbir}). Then $B$ is a section 
of $\bar f$ 
over a Zariski open subset in $\hat X$. Since the general fiber is a smooth rational curve and $\hat X$ is rational,
the variety $\bar X$ is rational as well.
\end{proof} 

\begin{lemma}
\label{lemma:hat-q}
In the notation of \xref{subsect:SL} assume that $\qW(X)=\qQ(X)$, $\p_1(X)\ge 2$, $s_1=0$, and $X$ is not rational.
Then $\hat q=1$ and there is an embedding $\Clt{\hat X}\subset \Clt{X}$. 
\begin{enumerate} 
\item 
If $\Clt{\hat X}=0$, then $\hat X\simeq \PP^1$, $\PP^2$, or $\PP(1,1,2)$. 
\item
If $\Clt{\hat X}\neq 0$, then $\hat X$ is a del Pezzo surface of degree $1$ and $|\Clt{\hat X}|\ge 3$. 
\end{enumerate}
The numbers $\p_i(X)$ satisfy the following conditions:
\begin{table}[H]
\setcounter{NN}{0}
\renewcommand{\arraystretch}{1.2}
\rm
\begin{tabularx}{0.8\textwidth}{|l|X|l|l|l|}
\hline 
&\multicolumn{1}{c|}{$\hat X$} & $\p_1(X)$ & $\p_2(X)$ & $\p_3(X)$ 
\\\hline
\nr &$\hat X\simeq \PP^1$ & $2$ & $\ge 3$ & $\ge 4$ \\ 
\nr & $\hat X\simeq \PP^2$ & $3$ & $\ge 6$ & $\ge 10$ \\ 
\nr & $\hat X\simeq \PP(1,1,2)$ & $2$ & $\ge 4$ & $\ge 6$ \\
\nr &$\hat X$ is a del Pezzo surface of degree $1$& $2$ &$\ge 4$ & $\ge 7$
\\\hline 
\end{tabularx} 
\end{table} 
\noindent
Furthermore, assume that $\Clt{\hat X}=0$.
If $s_2=0$ \textup(resp., $s_3=0$\textup), then equalities hold for $\p_2(X)$ \textup(resp., for $\p_3(X)$\textup).
\end{lemma}

\begin{proof}
We have $\hat q=1$ by Lemma~\ref{lemma:q=1}.
For short, we consider only the case where $\hat X$ is a surface. 
The case where $\hat X$ is a curve is much easier.
By Lemma~\ref{lemma:base-surface} we have $\Clt{\hat X}\subset \Clt{X}$ and $\hat X$ is a del Pezzo surface with Du Val singularities of type \type{A} and $\uprho(\hat X)=1$.
The pull-back map $\bar f^*$ of Weil divisors is well-defined and injective
(see the proof of Lemma~\ref{lemma:base-surface}).
Hence $\bar \MMM=\bar f^* |A_{\hat{X}}|$ for a primitive element $A_{\hat{X}}\in \Cl(\hat X)$
and $\dim |A_{\hat{X}}|=\dim \MMM=\dim |A_X|\ge 1$. If $K_{\hat X}^2>6$, then 
$\hat X$ is either $\PP^2$ or $\PP(1^2,2)$ (see e.g. \cite[Theorem~3.4]{Hidaka-Watanabe})
and $\p_1(X)=3$ or $2$ in these cases, respectively. Let
$K_{\hat X}^2\le 6$. Then $K_{\hat X}^2=1$ and $A_{\hat{X}}\sim -K_{\hat X}$
by Lemma~\ref{lemma:DP:0}. Hence $\p_1(X)=\dim |\bar \MMM|+1=2$.
Finally, $|\Clt{\hat X}|\ge 3$ by the classification \cite{Miyanishi-Zhang:88}. 

Note that $k \bar \MMM\sim \bar f^*(kA_{\hat{X}})$ for any $k$, hence the linear system 
$f_* \chi^{-1}_* | \bar f^*(kA_{\hat{X}})|$ is contained in $|kA_X|=\MMM_k$.
This implies that $\dim |kA_{\hat{X}}|\le \dim |kA_X|$. Then the inequalities in the table follow from 
Lemma~\ref{lemma:DP} and the Riemann-Roch for $-kK_{\hat X}$ in the case $K_{\hat X}^2=1$.
Finally, if $s_k=0$ and $\Clt{\hat X}=0$, then the linear system $\bar \MMM_k$ is $\bar f$-vertical, that is, 
$\bar \MMM_k=\bar f^* (kA_{\hat{X}})$ and so $\dim |kA_{\hat{X}}|= \dim |kA_X|$.
Hence, the inequalities in the table are in fact equalities in this case.
\end{proof}

\section{Case $\p_1(X)\ge 2$.}
\begin{setup}
\label{constr1}
Let $X$ be a non-rational $\QQ$-Fano threefold with $\qQ(X)=\qW(X)>1$ and $\p_1(X)\ge 2$.
We assume that $X$ is has at least one non-Gorenstein point.
This holds automatically if $\qQ(X)\ge 3$ because $X$ is not rational. 
The linear system
$|A_X|$ has no fixed components.
Apply the construction \eqref{diagram-main} with $n=1$, i.e. $\MMM=|A_X|$.
The relation \eqref{eq:main} for $k=1$ has the form
\begin{equation}
\label{equation-main:k=1}
\hat q=q s_1+(q\beta_1-\alpha)e,
\end{equation}
where $q\beta_1-\alpha$ is a positive integer by Corollary~\ref{cor-eq:main}.
Taking Lemma~\ref{lemma:q=1} into account we obtain two possibilities:
\subsubsection{Case $s_1>0$} 
\label{constr1a}
Then $\bar f$ is birational and $\hat q \ge q+1$.
\subsubsection{Case $s_1=0$} 
\label{constr1b}
Then $\bar f$ is a fibration and $\hat q=e=q\beta_1-\alpha=1$.
\end{setup} 

\begin{proposition}
\label{prop5}
Let $X$ be a $\QQ$-Fano threefold with $\qQ(X)\ge 5$. If $\p_1(X)\ge 2$, then $X$ is rational.
\end{proposition}
\begin{proof}
By Corollary~\ref{cor:q6-7} we may assume that
$\qQ(X)=5$. 
By Proposition~\ref{prop:tor} the group $\Cl(X)$ is torsion free.
Apply the construction \eqref{diagram-main} with $\MMM=|A_X|$ (see~\ref{constr1}).
Assume that $X$ is not rational. 

If $s_1>0$, then $\hat q \ge 6$ (see~\ref{constr1a}).
By Theorem~\ref{thm:q8} and Corollary~\ref{cor:q6-7} we have $s_1\ge 2$.
Hence $\hat q \ge 11$ and $X$ is rational by Theorem~\ref{thm:q8}, a contradiction.

Therefore, $s_1=0$, $\bar f$ is fibration, $\hat q=e=1$, and $5\beta_1=\alpha+1$ 
(see~\ref{constr1b}).
Let $P\in X$ be any point of index $r>1$. 
Since $\qQ(X)=\qW(X)$, the numbers $r$ and $5$ are coprime.
Take $m\in \ZZ_{>0}$ so that $5m\equiv 1\mod r$.
Then $m(-K_X)\sim \MMM$ near $P$ and so $\beta_1\ge m\alpha$ by Lemma~\ref{lemma:ct}.
We obtain $\alpha+1=5\beta_1\ge 5m\alpha$ and $\alpha\le \frac 1{5m-1}$.
Since $X$ is not rational, we can take $P$ so that $r\notin \{2,4\}$ by Theorem~\ref{thm:QFanoF}.
Then $5\not\equiv 1\mod r$, hence $m\ge 2$ and so $\alpha\le 1/9$.
This implies that $f(E)$ is a point of index $\ge 9$.
In this case computer search (see Sect.~\ref{sect:app}) gives the only possibility \cite[\# 41446]{GRD}:

\begin{quote}
$\B(X)=(2,9)$, \quad $A_X^3=5/18$, \quad $\p_1(X)=2$, \quad $\p_2(X)=4$, \quad $\p_3(X)=7$.
\end{quote}
Then $\alpha=\beta_1=1/9$ because $\B(X)=(2,9)$.
Recall that $e=1$. 
Then one can show that \eqref{eq:main} implies $s_2=s_3=0$.
This contradicts Lemma~\ref{lemma:hat-q}.
\end{proof}

\begin{proposition}
\label{prop4}
Let $X$ be a $\QQ$-Fano threefold with $\qQ(X)\ge 4$. If $\p_1(X)\ge 2$, then $X$ is rational.
\end{proposition}

\begin{proof}
By Proposition~\ref{prop5} we may assume that $\qQ(X)=4$. 
By Proposition~\ref{prop:F-tor4} the group $\Cl(X)$ is torsion free.
Apply the construction \eqref{diagram-main} with $\MMM=|A_X|$.
Assume that $X$ is not rational.

If $s_1>0$, then $\hat q \ge 5$ (see~\ref{constr1a}), hence $s_1\ge 2$ by
Theorem~\ref{thm:q8}, Corollary~\ref{cor:q6-7}, and Proposition~\ref{prop5}.
In this case \eqref{equation-main:k=1} implies $\hat q\ge 9$, hence $X$ is rational, a contradiction.

Therefore, $s_1=0$, $\bar f$ is a fibration, $\hat q=1$,
and $e=4\beta_1-\alpha=1$.
Let $P\in X$ be a point of maximal index $r$. Recall that $r$ must be odd because $\qW(X)=4$.
Take $m$ so that $4m\equiv 1\mod r$ and $0<m<r$, that is,
\[
m=
\begin{cases}
(3r+1)/4& \text{if $r\equiv 1\mod 4$},
\\
(r+1)/4 & \text{if $r\equiv -1\mod 4$}.
\end{cases}
\]
In both cases, $4m-1\ge r$.
By Lemma~\ref{lemma:ct} we have
$\beta_1\ge m\alpha$, hence 
\[ 
\alpha+1=4\beta_1\ge 4m\alpha,\qquad \frac{1}{4m-1} \ge \alpha \ge \frac 1{r}, \qquad r\ge 4m-1. 
\]
On the other hand, $4m-1\ge r$. 
Therefore, $r=4m-1$ and $\alpha=1/r$.
We may assume that $f(E)=P$.
If $X$ has a point $P'$ of index $r'$ with $r'\equiv 1\mod 4$, then similar computations
show that 
$1/r+1=\alpha+1\ge 4m'\alpha = (3r'+1)/r$ and so $3r'\le r$.

Thus the variety $X$ and the point $f(E)$ satisfy the following properties:
\begin{enumerate}
\renewcommand{\theenumi}{\rm (\arabic{enumi})}
\renewcommand\labelenumi{\rm (\arabic{enumi})}
\item \label{enu:q4b}
$\p_1(X)\ge 2$ and $\g(X)\le 21$ \cite[Theorem~1.2]{P:2013-fano};

\item\label{enu:q4c}
$f(E)$ is a point whose index $r$ is maximal, $r\equiv -1\mod 4$ and $r\ge 7$;
\item \label{enu:q4d}
if $P\in X$ is a point of index $r'$ with $r'\equiv 1\mod 4$, then $3r'\le r$.
\end{enumerate}
Applying computer search (see Sect.~\ref{sect:app}) under these conditions we obtain the following possibilities: 
\begin{table}[H]
\setcounter{NN}{0}
\renewcommand{\arraystretch}{1.2}
\begin{tabularx}{0.8\textwidth}{|l|X|X|X|l|l|l|l|l|}
\hline
&$A_X^3$ & $\B(X)$ & $\g(X)$ & $\p_1(X)$ & $\p_2(X)$ & $\p_3(X)$ &\cite{GRD} 
\\\hline
\nr \label{tabular:a7} & $\fracf {3}{7}$ &
$(7)$ &
$14$ &
$2 $ & $ 5 $ & $ 9 $ & \# 41372
\\
\nr \label{tabular:b37} & $\fracf { 8}{21}$ &
$(3, 7)$ &
$12$ &
$2 $ & $ 4 $ & $ 8 $ & \# 41367
\\
\nr \label{tabular:c77} & $\fracf { 3}{7}$ &
$( 7^2)$ &
$13$ &
$2 $ & $ 4 $ & $ 8 $ & \# 41370
\\
\nr \label{tabular:d9} & $\fracf { 5}{9}$ &
$(9)$ &
$18$ &
$2 $ & $ 6 $ & $ 11 $ & \# 41381
\\
\nr \label{tabular:e311} & $\fracf { 13}{33}$ &
$(3, 11)$ &
$12$ &
$2 $ & $ 4 $ & $ 8 $ & \# 41368
\\\hline 
\end{tabularx} 
\end{table}
The relation \eqref{eq:main} for $k=2$ and $k=3$ has the form
\[
2=2\hat q=4s_2+4\beta_2-2\alpha,\qquad 3=3\hat q=4s_3+4\beta_3-3\alpha.
\]
Since $X$ has no points of index $2$, we have $\beta_2>0$ and so $\beta_2\ge 1/r=\alpha$.
We obtain $s_2=0$. Then the cases~\ref{tabular:a7} and~\ref{tabular:d9}
$\B(X)=(7)$ and $(9)$ are impossible by Lemma~\ref{lemma:hat-q} and our assumptions.
In the remaining cases~\ref{tabular:b37},~\ref{tabular:c77} and~\ref{tabular:e311}
again by Lemma~\ref{lemma:hat-q} we have $s_3\neq 0$, hence $\beta_3=0$ and $\alpha=1/3$.
Thus $f(E)$ is a point of index $3$. 
This contradicts the property~\ref{enu:q4c}.
\end{proof}

\begin{proposition}
\label{prop:3}
Let $X$ be a $\QQ$-Fano threefold with $\qQ(X)=3$. 
\begin{enumerate} 
\item 
If $\p_1(X)\ge 3$, then $X$ is rational. 
\item 
If $\p_1(X)= 2$ and $X$ is not rational, then $\Cl(X)$ is torsion free and one of the following holds:

\begin{enumerate}
\item\label{prop:3-2o}
$A_X^3=1/2$, $\B(X)=(2,2,2)$, $\g(X)= 7$, \cite[\# 41198]{GRD};

\item\label{prop:3-2a}
$A_X^3=2/5$, $\B(X)=(5)$, $\g(X)= 6$, \cite[\# 41195]{GRD};

\item\label{prop:3-2b}
$A_X^3=6/11$, $\B(X)=(11)$, $\g(X)= 7$, \cite[\# 41196]{GRD};

\item\label{prop:3-2c}
$A_X^3=10/17$, $\B(X)=(17)$, $\g(X)= 7$, \cite[\# 41197]{GRD}.

\end{enumerate}
In these cases $X$ is unirational and has a conic bundle structure.
\end{enumerate}
\end{proposition}

\begin{sremark}
By Theorem~\ref{thm:QFanoF} a variety of type~\ref {prop:3-2o} is a hypersurface
$X_6\subset \PP(1^2,2^2,3)$. It is known that a very general variety in the family is 
not rational \cite{Okada2019}. 
We do not know if the varieties satisfying~\ref {prop:3-2a},~\ref {prop:3-2b},
and~\ref {prop:3-2c} are rational or not. Moreover, we do not know if the varieties satisfying 
\ref {prop:3-2b} and~\ref {prop:3-2c} really exist. The general complete intersection 
$X_{6,6}\subset \PP(1^2,2,3^2,5)$ satisfies~\ref {prop:3-2a} but still we do not know if 
this is the only example.
\end{sremark}

\begin{proof}
Assume that $X$ is not rational and $\p_1(X)\ge 2$.
By Proposition~\ref{prop:F-tor4} we have $\qQ(X)=\qW(X)$
and $\dim |A_X|\ge 1$ by Proposition~\ref{prop:tq5}.
Apply the construction \eqref{diagram-main} with $\MMM=|A_X|$ (see~\ref{constr1}).

If $s_1>0$, then $\hat q \ge 4$ (see~\ref{constr1a}), hence $s_1\ge 2$ by
Proposition~\ref{prop4}.
Then $\hat q=7$ and by Corollary~\ref{cor:q6-7} we have $s_1\ge 3$ and so $\hat q>7$, a contradiction.

Therefore, $s_1=0$, $\hat q=e=1$,
and $3\beta_1=\alpha+1$, $\bar f$ is a fibration, and $\bar \MMM=\bar f^* |A_{\hat X}|$.
Since $\beta_1\ge \alpha$ by Lemma~\ref{lemma:ct}, we have $\alpha\le 1/2$.
In particular, this implies that $f(E)$ is non-Gorenstein point.

Assume that $\Clt{X}\neq 0$. Then $|\Clt{X}|\ge 3$ by Lemma~\ref{lemma:hat-q}.
Since $\qQ(X)=\qW(X)$, we have $|\Clt{X}|\neq 3$ by Theorem~\ref{thm:index}\ref{thm:indexb}.
Then $X$ is rational by Proposition~\ref{prop:tor:new}, a contradiction. 
Thus we may assume that $\Cl(X)$ is torsion free.

Let $P$ be a point of maximal index $r$. Then $r$ is not divisible by $3$ by Theorem~\ref{thm:index}\ref{thm:indexb}.

Consider the case where $r=2$. Then $X$ is as in~\ref{prop:3-2o}
by Theorem~\ref{thm:QFanoF},~\ref{thm:QFanoF:9} and we may assume that $f(E)=P$. 
Moreover, $\alpha=\beta_1=1/2$ and $s_2=1$ by \eqref{eq:main}. 
Note that in this case the linear system $\MMM_2=|2A_X|$ is base point free and 
$\dim \MMM_2=4$.
Let $\MMM_2'\subset \MMM_2$ be the subsystem consisting 
of all divisors passing through $P$. Then $\dim \MMM'_2=3$.
Similar to \eqref{eq:main} we have
\[
2= k\hat{q}=q s_k'+(q \beta_k'-k \alpha) e=3 s_2'+3 \beta_2'-1
\]
where $\beta_2'$ is a positive integer. Hence, $s_2'=0$, that is, $\bar \MMM_2'\subset \bar f^*|2A_{\hat X}|$.
Thus $\dim |A_{\hat X}|=1$ and $\dim |2A_{\hat X}| \ge 3$. Then $\hat X\not\simeq \PP^1$,
hence $\hat X$ is a surface and $\bar f$ is a $\QQ$-conic bundle. 
A general member $S\in \MMM_2$ is a smooth del Pezzo surface of degree $1$ and 
its proper transform $\bar S\subset \bar X$ is a rational multisection of $\bar f$
(because $s_2=1$).
This implies that $\bar X$ is unirational.
Unirationality of $X$ in this case is also proved in \cite{CampanaFlenner}.

From now on we assume that $r\ge 4$. Put
\[
m=:
\begin{cases}
(2r+1)/3& \text{if $r\equiv 1\mod 3$},
\\
(r+1)/3& \text{if $r\equiv -1\mod 3$}.
\end{cases}
\]
Note that in both cases $r\le 3m-1$.
Then $\beta_1\ge m\alpha$ by Lemma~\ref{lemma:ct}.
Hence,
\[
\alpha+1 =3\beta_1\ge 3m\alpha,\qquad \frac 1{3m-1}\ge \alpha\ge \frac 1r,\qquad r\ge 3m-1.
\]
Since $r\le 3m-1$, we obtain $r=3m-1$ and $\alpha=1/r$.
Thus we may assume that $f(E)=P$.
If $X$ has a point $P'$ of index $r'$ with $r'\equiv 1\mod 3$, then similar computations
show that 
$1/r+1=\alpha+1\ge 3m'\alpha = (2r'+1)/r$ and so $2r'\le r$.
Thus $X$ satisfies the following properties:

\begin{enumerate}
\renewcommand{\theenumi}{\rm (\arabic{enumi})}
\renewcommand\labelenumi{\rm (\arabic{enumi})}
\item 
$\p_1(X)\ge 2$ and $\g(X)\le 20$ \cite[Theorem~1.2]{P:2013-fano};
\item
if $r$ is the maximal index $r$ of points on $X$, then $r\equiv -1\mod 3$ and $r\ge 5$;
\item
if $X$ has a point of index $r'$ with $r'\equiv 1\mod 3$, then $2r'\le r$.
\end{enumerate}

Then computer search (see Sect.~\ref{sect:app}) under these conditions gives the following possibilities:

\begin{table}[H]
\setcounter{NN}{0}
\renewcommand{\arraystretch}{1.2}
\begin{tabularx}{0.98\textwidth}{|l|lXlll||l|lXlll|}
\hline
&$A_X^3$ & $\B(X)$ & $\g(X)$ & $\p_1(X)$ & $\p_2(X)$ & &$A_X^3$ & $\B(X)$ & $\g(X)$ & $\p_1(X)$ & $\p_2(X)$ 
\\\hline
\nr\label{tab:q3:2/5} & $2/5$ & $(5)$ & $6$ & $2 $ & $ 4 $
&
\nr & $11/10$ & $(2, 5)$ & $15$ & $3 $ & $ 8 $
\\
\nr & $3/5$ & $( 2^2, 5)$ & $8$ & $2 $ & $ 5 $&
\nr & $6/5$ & $( 5^2)$ & $16$ & $3 $ & $ 8 $
\\
\nr & $7/10$ & $(2, 5^2)$ & $9$ & $2 $ & $ 5 $
&\nr &$9/8$ & $(2, 8)$ & $15$ & $3 $ & $ 8 $
\\
\nr & $4/5$ & $( 5^3)$ & $10$ & $2 $ & $ 5 $
&\nr &$7/8$ & $( 2^2, 4, 8)$ & $11$ & $2 $ & $ 6 $
\\
\nr & $5/8$ & $( 2^2, 8)$ & $8$ & $2 $ & $ 5 $&
\nr & $29/40$ & $(2, 5, 8)$ & $9$ & $2 $ & $ 5 $
\\
\nr & $49/40$ & $(5, 8)$ & $16$ & $3 $ & $ 8 $&
\nr & $25/22$ & $(2, 11)$ & $15$ & $3 $ & $ 8 $
\\
\nr\label{tab:q3:6/11} & $6/11$ & $(11)$ & $7$ & $2 $ & $ 4 $
&\nr & $35/44$ & $(4, 11)$ & $10$ & $2 $ & $ 5 $
\\
\nr & $7/11$ & $( 2^2, 11)$ & $8$ & $2 $ & $ 5 $
&\nr & $11/14$ & $(14)$ & $10$ & $2 $ & $ 5 $
\\
\nr & $81/110$ & $(2, 5, 11)$ & $9$ & $2 $ & $ 5 $
&\nr\label{tab:q3:10/17} & $10/17$ & $(17)$ & $7$ & $2 $ & $ 4 $
\\
\nr & $9/14$ & $( 2^2, 14)$ & $8$ & $2 $ & $ 5 $ 
&&&&&&
\\\hline
\end{tabularx}
\end{table}
Note that $|2A_X|$ has no fixed components in our case.
Then one can show that \eqref{eq:main} implies $s_2=0$.
Hence by Lemma~\ref{lemma:hat-q} the only cases 
\ref{tab:q3:2/5},~\ref{tab:q3:6/11}, and~\ref{tab:q3:10/17} are possible.
We get 
\ref {prop:3-2a},~\ref {prop:3-2b},
and~\ref {prop:3-2c}. Note that in all cases $X$ has only cyclic quotient singularities.
By Theorem~\ref{thm:K:blowup} the $f$-exceptional divisor $E$ is toric, in particular, rational.
Its proper transform $\bar E\subset \bar X$ is a multisection of $\bar f$.
Hence $\bar X$ is unirational. This finishes the proof of Proposition~\ref{prop:3}.
\end{proof}

\section{$\QQ$-Fano threefolds with $\qQ(X)\ge 5$}
\begin{proposition}
\label{prop6}
Let $X$ be a $\QQ$-Fano threefold with $\qQ(X)\ge 6$. If $\p_2(X)\ge 2$, then $X$ is rational.
\end{proposition}

\begin{proof}
The group $\Cl(X)$ is torsion free by Proposition~\ref{prop:tor} and
by Corollary~\ref{cor:q6-7} we may assume that
$\qQ(X)=6$. 
Assume that $X$ is not rational. 
By Theorem~\ref{thm:QFanoF} the global Gorenstein index of $X$ is at least $6$.
Apply the computer search (see Sect.~\ref{sect:app}) or \cite{GRD} under the assumption $\p_2(X)\ge 2$
and $\p_1(X)\le 0$. 
We obtain the only possibility \cite[\#~41466]{GRD}:
\[
A_X^3= 3/35, \quad \B(X)= (5, 7), \quad \g(X)=9,\quad \h_X(t)=1+t+2 t^2+3t^3+\cdots.
\]
Since $\dim |A_X|=0$, the linear system $|2A_X|$ has no fixed components.
Hence we can apply the construction \eqref{diagram-main} with $\MMM=|2A_X|$.
In a neighborhood of the point of index $7$ we have $\MMM \sim 5(-K_X)$ and so $\beta_2\ge 5\alpha$
by Lemma~\ref{lemma:ct}.
The relation~\eqref{eq:main} for $k=2$ has the form 
\[
\hat{q}=3s_2+(3\beta_2- \alpha) e\ge 3s_2+14\alpha e.
\]
Since $\alpha\ge 1/7$, we see that $\hat q \ge 2$. 
Then the contraction $\bar f$ is birational by Lemma~\ref{lemma:q=1}.
Since $\dim \MMM_2>0$, we have $s_2>0$ and so $\hat q \ge 5$.
Then $s_2\ge 2$ by Proposition~\ref{prop4}.
Hence $\hat q \ge 8$ and $X$ is rational by Theorem~\ref{thm:q8}, a contradiction.
\end{proof}

\begin{proposition}
\label{prop6a}
Let $X$ be a $\QQ$-Fano threefold with $\qQ(X)=6$. If $\p_3(X)\ge 2$, then $X$ is rational.
\end{proposition}
\begin{proof}
Assume that $X$ is not rational.
Assume also that $\p_2(X)\le 1$ and $\p_3(X)\ge 2$.
Applying computer search (see Sect.~\ref{sect:app} or \cite{GRD}) we obtain the only possibility
\cite[\# 41462]{GRD}:
\[
A_X^3= 2/35, \quad \B(X)= (5, 7^2), \quad \g(X)=5,\quad \h_X(t)=1+t^2+2t^3+3t^4+\cdots.
\]
Apply the construction \eqref{diagram-main} with $\MMM=|3A_X|$.
In a neighborhood of the point of index $7$ we have $\MMM \sim 4(-K_X)$ and so $\beta_3\ge 4\alpha$
by Lemma~\ref{lemma:ct}.
The relation~\eqref{eq:main} for $k=3$ has the form
\[
\hat{q}=2s_3+(2\beta_3- \alpha) e\ge 2s_3+7\alpha e.
\]
We claim that $\alpha=1/7$ and $e=1$.
Indeed, otherwise $\hat q\ge 7\alpha e>1$.
Hence $\bar f$ is birational by Lemma~\ref{lemma:q=1}.
In this case, $s_3>0$ and so $\hat q\ge 4$.
Then $s_3\ge 2$ by Proposition~\ref{prop4}
and so $\hat q \ge 6$.
From Proposition~\ref{prop6} we obtain $s_3\ge 3$
and so $\hat q\ge 8$. This contradicts Theorem~\ref{thm:q8}.

Thus $\alpha=1/7$ and $e=1$. Then we consider~\eqref{eq:main} for $k=2$.
Since $f(E)$ is a point of index~$7$, the number
$7\beta_2$ is an integer and $7\hat{q}=3(7s_2+7\beta_2)-1$, hence
$\hat q\equiv -1\mod 3$. In particular, $\hat q\neq 1$, hence $\bar f$ is birational and $s_2,\, s_3>0$.
Then the only possibility is $\hat q=5$ and $s_3=2$ by Proposition~\ref{prop4}.
Hence $\p_2(\hat X)\ge 2$. 
Since $e=1$ and $|A_X|=\varnothing$, we have $\Clt{X}\neq 0$ by Lemma~\ref{lemma:sl:torsion}.
Then $\hat X$ is rational by Proposition~\ref{prop:tq5}.
\end{proof} 

\begin{proposition}
\label{prop7}
Let $X$ be a $\QQ$-Fano threefold with $\qQ(X)=7$. If $\p_3(X)\ge 2$, then $X$ is rational.
\end{proposition}

\begin{proof}
Assume that $X$ is not rational. 
By Proposition~\ref{prop:qW-neq-qQ} we have $\qQ(X)=\qW(X)$ and $\p_2(X)\le 1$ by Corollary~\ref{cor:q6-7}.
Computer search (see Sect.~\ref{sect:app} or \cite{GRD}) produces four numerical possibilities which will be considered below.
We will use the construction \eqref{diagram-main} with $\MMM=|3A_X|$ or $\MMM=|4A_X|$.
The relation~\eqref{eq:main} for $k=3$ and $4$ has the form
\begin{eqnarray}
\label{eq:q=7:3}
3 \hat{q}&=&7s_3+(7 \beta_3-3 \alpha) e,
\\
\label{eq:q=7:4}
4\hat{q}&=&7s_4+(7\beta_4- 4\alpha) e.
\end{eqnarray} 

\subsection*{Case $A_X^3= 1/24$, \cite[\# 41477]{GRD} } Then $\B(X)=(2^2, 3, 4, 8)$ and $\h_X(t)=1+t^2+2 t^3+3 t^4+\cdots$.
Thus $\dim |3A_X|=1$. 
Note that in this case $X$ can have a $2$-torsion in $\Cl(X)$ (see Proposition~\ref{prop:tor}).
Apply the construction \eqref{diagram-main} with $\MMM=|3A_X|$.
In a neighborhood of the point of index $8$ we have $\MMM \sim 5(-K_X)$ and so $\beta_3\ge 5\alpha$
by Lemma~\ref{lemma:ct}.
By \eqref{eq:q=7:3} 
\[
3\hat{q}=7s_3+(7\beta_3- 3\alpha) e\ge 7s_3+32\alpha e\ge 7s_3+4e.
\]
Hence, $\hat q\ge 2$ and $\bar f$ is birational. Then $s_3>0$ and $\hat q\ge 4$.
By Proposition~\ref{prop4} we have $s_3\ge 2$ and $\hat q\ge 6$.
By Proposition~\ref{prop6} we have $s_3\ge 3$ and $\hat q\ge 9$.
This contradicts Theorem~\ref{thm:q8}. 

\subsection*{Case $A_X^3= 1/18$, \cite[\# 41480]{GRD}} Then $\B(X)= (3, 6, 9)$, the group $\Cl(X)$ is torsion free, and and $\h_X(t)=1+t+t^2+2 t^3+3 t^4+\cdots$. Thus $\dim |4A_X|=2$.
Apply the construction \eqref{diagram-main} with $\MMM=|4A_X|$.
In a neighborhood of the point of index $9$ we have $\MMM \sim 7(-K_X)$ and so $\beta_4\ge 7\alpha$
by Lemma~\ref{lemma:ct}. 
From \eqref{eq:q=7:4} we have
\[
4\hat{q}=7s_4+(7\beta_4- 4\alpha) e\ge 7s_4+45\alpha e \ge 7s_4+5e.
\]
Hence $\hat q\ge 2$, $\bar f$ is birational, $s_4\ge 1$, and $\hat q\ge 3$. 
Then $\alpha<1$, so $f(E)$ is a point of index $r=3$, $6$ or $9$ and $\alpha=1/r$ (see Theorem~\ref{thm:K:blowup}).
The relation~\eqref{eq:main} for $k=1$ has the form
$\hat{q}=7s_1+(7 \beta_1- \alpha) e$,
where $\beta_1\ge \frac 14\beta_4\ge \frac 74\alpha$ because $4M_1\in \MMM$. This gives us $s_1=0$, $e=1$ 
by Lemma~\ref{lemma:sl:torsion}, and $\hat q=7 \beta_1- \alpha$.
If $\hat q\ge 6$, then $s_4\ge 3$ by Proposition~\ref{prop6}, and $\hat q>7$, a contradiction.

Thus $3\le \hat q\le 5$.
If $\hat q=3$, then $s_4=e=1$ and $\alpha=1/9$.
Thus $\hat \MMM_4\subset |A_{\hat X}|$.
We obtain $\dim |A_{\hat X}|\ge 2$.
This contradicts Proposition~\ref{prop:3}.
Therefore, $\hat q\ge 4$, then $s_3\ge 2$ by Proposition~\ref{prop4}, hence $\hat q=5$ by \eqref{eq:q=7:3}.
Then $\beta_1=(5+\alpha)/7=(5r+1)/7r$, so $5r+1\equiv 0\mod 7$.
This is contradicts $r\in \{3,6,9\}$. 

\subsection*{Case $A_X^3= 1/33$, \cite[\# 41476]{GRD}} 
Then $\B(X)= (2^2, 3, 11)$, the group $\Cl(X)$ is torsion free, and $\h_X(t)=1+t^2+2 t^3+2 t^4+\cdots$.
Apply the construction \eqref{diagram-main} with $\MMM=|4A_X|$.
In a neighborhood of the point of index $11$ we have $\MMM \sim 10(-K_X)$ and so $\beta_4\ge 10\alpha$
by Lemma~\ref{lemma:ct}. By \eqref{eq:q=7:4} 
\[
4\hat{q}=7s_4+(7\beta_4- 4\alpha) e\ge 7s_4+66\alpha e \ge 7s_4+6e.
\]
Hence $\hat q\ge 2$, $\bar f$ is birational, and $s_4\ge 1$.
Then $\hat q\ge 5$, $s_4\ge 2$ by Proposition~\ref{prop4}, and $\alpha\le 7/33$, so 
$f(E)$ is a point of index $11$ and $\alpha=1/11$ (see Theorem~\ref{thm:K:blowup}).
If $\hat q\ge 6$, then $s_3,\, s_4\ge 3$ by Proposition~\ref{prop6}.
Since $\beta_3\ge \alpha=1/11$, from \eqref{eq:q=7:3} we obtain $\hat q>7$, a contradiction.
Therefore, $\hat q=5$, $s_3=s_4=2$, and $e=1$.
Hence $\p_1(\hat X)\ge 1$ and $\p_2(\hat X)\ge 2$.
On the other hand, $\Clt{\hat X}\neq 0$ by Lemma~\ref{lemma:sl:torsion} because $|A_X|=\varnothing$.
Then by
Proposition~\ref{prop:tq5} the variety $\hat X$ is rational.

\subsection*{Case $A_X^3= 1/30$, \cite[\# 41479]{GRD}} 
Then $\B(X)=(2, 6, 10)$ and 
\begin{equation*}
\h_X(t)=1+t+t^2+t^3+2 t^4+3 t^5+\cdots.
\end{equation*}
In particular, $\dim |3A_X|=0$ and $\dim |4A_X|=1$.
By our assumption $\p_3(X)\ge 2$, hence $X$ has to have a $2$-torsion $T\in \Cl(X)$ (see Proposition~\ref{prop:tor}).
Hence $X$ is of type~\ref{prop:tor}\ref {tab:h:q=7t:d=1/30}.
Apply the construction \eqref{diagram-main} with $\MMM=|4A_X|$.
In a neighborhood of the point of index $6$ we have $\MMM \sim 4(-K_X)$ and so $\beta_4\ge 4\alpha$
by Lemma~\ref{lemma:ct}.
By \eqref{eq:q=7:4}
\[
4\hat{q}=7s_4+(7\beta_4- 4\alpha) e\ge 7s_4+24\alpha e.
\]
If $\alpha=1$, then $\hat q\ge 6$ and so $s_4\ge 2$ by Proposition~\ref{prop4}.
But in this case $\hat q>7$, a contradiction. Thus $\alpha<1$ and so
$f(E)$ is a cyclic quotient singularity of index $r=2$, $6$ or $10$. 
In particular, $\alpha=1/r$ by Theorem~\ref{thm:K:blowup}.
\footnote{Note that modulo $\ZZ$ we have $\beta_4\equiv 0$ (resp. $\beta_4\equiv 2/3$, and $\beta_4\equiv 1/5$)
in the case $r=2$ (resp. $r=6$, $r=10$). Taking this into account we obtain the only solution: $\hat q=1$, $r=6$, $s_4=0$.}
The relation~\eqref{eq:main} for $k=1$ has the form
\[
\hat q=7s_1+(7\beta_1-\alpha)e.
\]
Here $\beta_1\ge \alpha$, hence $\hat q=1$ because $\hat q<8$ by Theorem~\ref{thm:q8}.
Then $s_1=s_4=0$, $r=6$, $\alpha=1/6$, and $e=1$. This means that $\bar f$ is a fibration, $\bar M_1\sim \bar f^* A_{\hat{X}}$
and $\bar\MMM_4=\bar f^* |4A_{\hat{X}}|$ for a primitive element $A_{\hat{X}}\in \Cl(\hat X)$.
Hence $\dim |A_{\hat{X}}|=0$ and $\dim |4A_{\hat{X}}|=1$.
This is impossible if $\hat X\simeq \PP^1$.
Therefore, $\hat X$ is a Du Val del Pezzo surface.
Thus $A_{\hat{X}}$ is a line on $\hat S$ by Lemma~\ref{lemma:DP:0}\ref{lemma:DP:0-3}.
Since $\dim |2A_{\hat{X}}|=\dim |3A_{\hat{X}}|=0$ we see that $\Clt{\hat X}\neq 0$ by Lemma~\ref{lemma:DP}.
Therefore, $\hat X$ contains a line $\hat L$ other than $A_{\hat{X}}$ (see Lemma~\ref{lemma:DP:0}\ref{lemma:DP:0-6}).
Let $\bar D:= \bar f^* \hat L$. Then $\bar D$ is an irreducible effective divisor and $\bar D\neq \bar M_1$.
Hence $\tilde D:=\chi^{-1}_*\bar D\neq \tilde M_1$ and $D:=f_*\tilde D\neq M_1$.
So, $D$ is an effective divisor on $X$ such that $D\not \simeq A_X$ but
$D\qq A_X$. This contradicts~Proposition~\ref{prop:tor}\ref {tab:h:q=7t:d=1/30}.
\end{proof}

\begin{scorollary}
\label{cor:ge6b}
Let $X$ be a $\QQ$-Fano threefold with $\qQ(X)\ge 6$. If either $\Clt{X}\neq 0$ or  $\p_3(X)\ge 2$, then $X$ is rational.
\end{scorollary} 

\begin{proposition}
\label{prop5a}
Let $X$ be a $\QQ$-Fano threefold with $\qQ(X)=5$ and $\p_2(X)\ge 2$.
Assume that $X$ is not rational. Then $X$ belongs to the following class:
\begin{enumerate}
\item[(*)]
\label{prop5a-b}
$\B(X)=(2^2, 3, 4)$, $A_X^3= 1/12$, $\Clt{X}= 0$, $\p_1(X)=1$, $\p_2(X)=2$,
\cite[\# 41422]{GRD}. 	 
\end{enumerate}
Moreover, $X$ is birational to a conic bundle, and 
if the point of index $4$ is a cyclic quotient singularity, then $X$ is 
unirational.
\end{proposition}

\begin{sremark}
A general hypersurface $X_{10}\subset \PP(1,2,3,4,5)$ belongs to the class (*)
and according to \cite{Okada2019} a very general such a hypersurface is not rational.
However we do not know that the family of such hypersurfaces exhaust (*).
\end{sremark} 

\begin{proof}
We have $\qQ(X)=\qW(X)=5$ by Proposition~\ref{prop:qW-neq-qQ}.

\begin{sclaim}
$|A_X|\neq \varnothing$ and $\dim |2A_X|\ge 1$.
\end{sclaim}
\begin{proof}
If $\Clt{X}\neq 0$, the assertions follow from Proposition~\ref{prop:tor}.
Thus we may assume that $\Cl(X)\simeq \ZZ\cdot A_X$. Then $\dim |2A_X|\ge 1$ by our assumption $\p_2(X)\ge 2$.
Computer search shows that there are 35 Hilbert series of $\QQ$-Fano threefolds with $\qW(X)=5$ and $\Clt{X}=0$,
and in all cases $\dim |2A_X|\ge 1$ implies $|A_X|\neq \varnothing$ (see also \cite{GRD}). 
\end{proof}

\begin{sclaim}
\label{claim:q=5}
We have $s_1=0$, $e=1$, and one of the following holds:
\begin{enumerate}
\item \label{claim:q=5a}
$s_2=0$, $\hat q=1$, $5\beta_1=\alpha+2s_2+1$, 
$5\beta_2=2\alpha+2$, $\bar f$ is a fibration, or
\item \label{claim:q=5b}
$s_2=1$, $\hat q=3$, $5\beta_1=\alpha+3$, 
$5\beta_2=2\alpha+1$, and $\bar f$ is birational.
\end{enumerate}
Moreover, in the case~\ref{claim:q=5a} we have $s_3=0$ if $\beta_3>0$ and $s_4=0$ if $\beta_4>0$,
and in the case~\ref{claim:q=5b} we have $s_3=1$ if $\beta_3>0$.
\end{sclaim}

\begin{proof}
The relation~\eqref{eq:main} for $k=1$ and $2$ has the form
\begin{eqnarray*}
\hat{q}&=&5s_1+(5\beta_1- \alpha) e,
\\
2\hat{q}&=&5s_2+(5 \beta_2-2 \alpha) e.
\end{eqnarray*}
Note that $5\beta_1\ge \frac 52\beta_2 >\alpha$. Hence $\hat q> 5s_1$.
If $s_1>0$, then $\hat q\ge 6$ and $s_2\ge 3$ by Proposition~\ref{prop6}.
But then $\hat q> 15$, a contradiction. Therefore, $s_1=0$ and 
$\hat{q}=(5\beta_1- \alpha) e$. 

Consider the case 
$s_2=0$. Then $\hat q=1$. Since $5\beta_1- \alpha$ is an integer, we have 
$5\beta_1=\alpha+1$, $e=1$, and so $5 \beta_2=2 \alpha+2$. 

Consider the case $s_2>0$, then $\hat q\ge 3$. 
If moreover $\hat q>3$, then $s_2\ge 2$ by Proposition~\ref{prop4}.
In this case $\hat q>5$ and $s_2\ge 3$ by Proposition~\ref{prop6}.
But then $\hat q>7$, a contradiction. Thus
$\hat q= 3$, then $s_2=5 \beta_2-2 \alpha=e=1$ and $5\beta_1=\alpha+3$. 

The last statement follows from~\eqref{eq:main}.
\end{proof}

\par\smallskip\noindent
\textit{Proof of Proposition~\xref{prop5a} \textup(continued\textup).} Thus we have
\[
\hat q=2s_2+1,\qquad 5\beta_1=\alpha+2s_2+1,\qquad 
5\beta_2=2\alpha+2-s_2. 
\]
Let $P\in X$ be a point of index $r>1$. 
Take $m$ so that $5m\equiv 2\mod r$ and $0<m<r$.
Then $\beta_2\ge m \alpha$ by Lemma~\ref{lemma:ct}. This gives us
$2\ge 2-s_2\ge 3m \alpha$.
By Theorem~\ref{thm:QFanoF} we may assume that $r>2$ and then $m\ge 1$ and $\alpha\le 2/3$.
Hence $f(E)$ is a non-Gorenstein point. Now take $P=f(E)$.
Then $b_i:=\beta_i r$ and $a:=\alpha r$ are integers such that
\[
5b_1=a+(2s_2+1)r, \qquad 5b_2=2a+(2-s_2)r\ge 5ma.
\]
Since $r\le 24$ by \eqref{eq:r}, we have $(5m-2)a\le 48$. Now it is easy to enumerate all the possibilities for 
the point $f(E)$ and numbers $\beta_i$, $s_2$, and $m$:

\begin{table}[H]
\renewcommand{\arraystretch}{1.2}
\begin{tabularx}{0.8\textwidth}{|m{0.07\textwidth}|m{0.07\textwidth}|m{0.07\textwidth}|m{0.07\textwidth}|m{0.07\textwidth}| m{0.07\textwidth}|m{0.07\textwidth}|X|}
\hline
$r$ & $s_2$ & $\hat q$ & $\alpha$ & $\beta_1$ & $\beta_2$&$m$&$\ct(X,\MMM)$
\\\hline
3& 1& 3& 1/3& 2/3& 1/3& 1& 1 \\
3& 0& 1& 2/3& 1/3& 2/3& 1& 1 \\
4& 0& 1& 1/4& 1/4& 1/2& 2& 1/2 \\
8& 1& 3& 1/8& 5/8& 1/4& 2& 1/2 \\
8& 0& 1& 1/4& 1/4& 1/2& 2& 1/2 \\
9& 0& 1& 1/9& 2/9& 4/9& 4& 1/4 \\\hline
\end{tabularx}
\end{table}
Let $P'\in X$ be a point $P'$ of index $r'>1$ and let $m'$ is an integer such that
$5m'\equiv 2\mod r'$ and $0<m'<r'$. Then $\ct(X,\MMM)\le 1/m'$ by Lemma~\ref{lemma:ct}
and so $m'\le m\le 4$. This shows that $X$ can contain only points of indices 
$r'=2$, $3$, $4$, $6$, $8$, $9$, $13$, $18$.
Computer search shows that under the assumptions $\p_2(X)\ge 2$ and $\p_1(X)\le 1$ we have 
$\B(X)= (2, 4^2, 6)$, $(2^2, 3, 9)$, $(2^3, 3, 8)$,$(2^3, 3^2)$, $(2^2, 4, 8)$, or $(2^2, 3, 4)$.
Consider these cases separately.

\subsection*{Case $\B(X)=(2, 4^2, 6)$, \cite[\# 41434]{GRD}} 
In this case $r=4$ and for $r'=6$ we have $m'=4>m=2$, a contradiction. 

\subsection*{Case $\B(X)=(2^2, 3, 9)$, \cite[\# 41423]{GRD}} 
Then 
the group $\Cl(X)$ is torsion free, 
\begin{equation}
\label{eq:q=5:dimA}
\dim | A_X|= 0,\quad 
\dim | 2 A_X|= 1,\quad 
\dim | 3 A_X|= 2,\quad \text{and}\quad 
\dim | 4 A_X|= 4.
\end{equation} 
In this case $r=9$ and $\hat q=1$.
Then $s_3=s_4=0$ by Claim~\ref{claim:q=5} because $3A_X$ and $4A_X$ are not Cartier at $P=f(E)$.
We see that $\bar \MMM_k=\bar f^* |\hat M_k|$ for $k=2,3,4$.
If $\hat X\simeq \PP^1$, then $\bar \MMM_2=\bar f^* |\OOO_{\PP^1}(1)|$
and $\bar \MMM_4=\bar f^* |\OOO_{\PP^1}(2)|$. This contradicts \eqref{eq:q=5:dimA}.
Hence by Lemma~\xref {lemma:base-surface} \ 
$\hat X$ is a Du Val del Pezzo surface with only type \type{A}-singularities and $\Cl(\hat X)\simeq \ZZ$. This contradicts Lemma~\xref{lemma:DP}. 

\subsection*{Case $\B(X)=(2^3, 3, 8)$, \cite[\# 41440]{GRD}} 
Then the group $\Cl(X)$ is torsion free, $r=8$, $\dim | A_X|= 0$, $
\dim | 2 A_X|= 2$, and $
\dim | 3 A_X|= 4$. If $s_2=1$, then $\hat q=3$, $\hat X$ is a $\QQ$-Fano with $\p_1(\hat X)\ge 3$. In this case $\hat X$ is rational by Proposition~\ref{prop:3}.
Let $s_2=0$. Then $s_3=0$ by Claim~\ref{claim:q=5} because $3A_X$ is not Cartier at $P=f(E)$.
Hence $\bar \MMM_k=\bar f^* |\hat M_k|$ for $k=2$ and $3$, where $\dim |\hat M_2| =\dim | 2 A_X|= 2$
and $\dim |\hat M_3| =\dim | 3 A_X|= 4$. As above, $\hat X\not \simeq \PP^1$ and we get a contradiction by Lemma~\xref{lemma:DP}.

\subsection*{Case $\B(X)=(2^3, 3^2)$, \cite[\# 41439]{GRD}} 
In this case $r=3$, the group $\Cl(X)$ is torsion free, $\dim |A_X|= 0$, and $\dim | 2 A_X|= 2$. 
If $\hat q=3$, then $\hat X$ ia a $\QQ$-Fano threefold with 
$\p_1(\hat X)\ge 3$ because $s_2=1$. Then $\hat X$ is rational by Proposition~\ref{prop:3}.
Let $\hat q=1$. Then $\bar f$ is a fibration such that $\bar \MMM_2=\bar f^* |\hat M_2|$ 
with $\dim |\hat M_2|=2$. Since $\dim |\bar M_1|=0$, we have $\hat X\not \simeq \PP^1$.
Hence $\hat X$ is a Du Val del Pezzo surface such that 
$\Cl(\hat X)\simeq \ZZ$, $\dim |A_{\hat X}|=0$ and $\dim |2A_{\hat X}|=2$.
This contradicts Lemma~\ref{lemma:DP}. 

\subsection*{Case $\B(X)=(2^2, 4, 8)$, \cite[\# 41425]{GRD}} 
Then $\dim |kA_X|= k-1$ for $k=1,2,3$.
If $\Clt{X}\neq 0$, then $X$ is rational by Proposition~\ref{prop:tq5}.
Thus we may assume that $\Cl(X)$ is torsion free.
Apply the construction \eqref{diagram-main} with $\MMM=|3A_X|$.
In a neighborhood of the point of index $8$ we have $\MMM \sim 7(-K_X)$ and so $\beta_3\ge 7\alpha$
by Lemma~\ref{lemma:ct}.
The relation~\eqref{eq:main} for $k=3$ has the form
\[
3\hat{q}=5s_3+(5\beta_3- 3\alpha) e\ge 5s_3+32\alpha e\ge 5s_3+4e.
\]
Hence $\hat q>1$, $\bar f$ is birational, $s_3>0$, and $\hat q\ge 3$.
Then $s_3\ge 2$ by Proposition~\ref{prop:3} because $\dim |3A_X|=2$.
Hence $\hat q\ge 5$.
If $\hat q\ge 6$, then $s_3\ge 4$ by Proposition~\ref{cor:ge6b} and $\hat q\ge 8$.
This contradicts Theorem~\ref{thm:q8}.
Thus $\hat q=5$, $e=1$, $s_3=2$, and $\alpha<1/4$.
Hence $\alpha=1/8$ and $\beta_3=43/40\notin \frac18 \ZZ$, a contradiction.

\subsection*{Case $\B(X)=(2^2, 3, 4)$, \cite[\# 41422]{GRD}} 
Then the group $\Cl(X)$ is torsion free, $\dim |kA_X|= k-1$ for $k=1,2,3$, 
$r=4$, and $\hat q=1$.
As above, we obtain $\hat X$ is a Du Val del Pezzo surface
and $\bar f$ is a $\QQ$-conic bundle, i.e. $X$ is in situation (*) of~\ref{prop5a}.
If $f(E)$ is a cyclic quotient singularity, then $E\simeq \PP(1,1,3)$ 
by Theorem~\ref{thm:K:blowup} and as in the proof of Proposition~\ref{prop:3}
we conclude that $\bar X$ is unirational. This finishes the proof of Proposition~\xref{prop5a}.
\end{proof} 

\section{$\QQ$-Fano threefolds with $\qQ(X)=2$}

The following proposition slightly improves the corresponding result in \cite{P:fano-conic}.
\begin{proposition}
\label{prop:2}
Let $X$ be a $\QQ$-Fano threefold with $\qQ(X)=2$. 
Assume that $X$ is not Gorenstein.
\begin{enumerate} 
\item \label{prop:2a}
If $\p_1(X)\ge 2$, then $X$ is not solid, i.e. it is birational to a strict Mori fiber space. 
\item \label{prop:2b}
If $\p_1(X)\ge 3$, then $X$ is birational to a conic bundle.
If furthermore non-Gorenstein points of $X$ are of types \type{cA/r}, then 
$X$ is unirational.
\item \label{prop:2c}
If $\p_1(X)\ge 4$, then $X$ is rational.
\end{enumerate}
\end{proposition}
\begin{proof}
Assume that $X$ is not rational and $\p_1(X)\ge 2$. 
Let $A'$ be a divisor such that $A'\qq A_X$ and $\dim |A'|=\p_1(X)-1$.
Apply the construction \eqref{diagram-main} with $\MMM:=|A'|$ (cf.~\ref{constr1}).
We have 
\begin{equation}
\label{eq:hatq-last}
\hat q=2 s_1+(2\beta_1-\alpha)e,
\end{equation}
where $2\beta_1-\alpha>0$ by Lemma~\ref{lemma:ct}. 

First, consider the case $s_1\ge 2$. Then $\hat q \ge 5$.
If $\hat q\ge 6$, then $s_1\ge 4$ by Corollary~\ref{cor:ge6b}.
But in this case $\hat q>7$, a contradiction. Therefore, 
$\hat q=5$. Then $s_1= 2$, $\p_2(\hat X)\ge 2$, and so $X$ is birational to a conic bundle by
Proposition~\ref{prop5a}. If furthermore $\p_1(X)\ge 3$, then $\p_2(\hat X)\ge 3$
and $X$ is rational again by Proposition~\ref{prop5a}.

Consider the case $s_1=1$. Then $\hat q \ge 3$, $\p_1(\hat X)\ge 2$, and 
$X$ is unirational and has a conic bundle structure by Propositions~\ref{prop4} and~\ref{prop:3}.
Moreover, if $\p_1(X)\ge 3$, then $X$ is rational.

Finally, consider the case $s_1=0$.
Then, $\hat q=1$ and $\bar f$ is a fibration. This proves~\ref{prop:2a}.
Now, assume that $\p_1(X)\ge 3$. Then $\hat X\simeq \PP^2$ (see Lemma~\ref{lemma:hat-q}) and $\bar f$ 
is a $\QQ$-conic bundle. Moreover, in this case we have $\p_1(X)=3$. This proves \ref{prop:2c}.
To prove \ref{prop:2b} we note that $2\beta_1\le \alpha+1$ 
by~\eqref{eq:hatq-last}. On the other hand, $\beta_1\ge \alpha$ by Lemma~\ref{lemma:ct}.
Hence $\alpha\le 1$. If $\alpha<1$, then $f(E)$ is a point of index $>1$. 
If $\alpha=1$, then $f(E)$ 
there is a canonical center of $(X,\MMM)$ 
which is a point of index $>1$ again by Lemma~\ref{lemma:ct}.
Thus replacing $f$ with another extremal blowup if necessary (see Remark~\ref{rem:SL:choice}) we may assume that $f(E)$ is a non-Gorenstein point.
By our assumption $f(E)$ must be a point of type \type{cA/r}.
In this case the divisor $E$ must be a rational surface \cite{Prokhorov-2002b}.
Its proper transform $\bar E\subset \bar X$ is a multisection of $\bar f$.
Hence $\bar X$ is unirational. This proves~\ref{prop:2b}.
\end{proof} 

\appendix 
\section{}
\label{sect:app}
\setcounter{subsection}{1}
In this section we present a computer algorithm (see \cite[\S~3]{P:2019:rat:Q-Fano} or \cite[\S~3]{Caravantes2008}) that alow to list 
all the numerical possibilities for $\QQ$-Fano threefolds of index at least $3$.
Let $X$ be a $\QQ$-Fano threefold 
with $q:=\qQ(X)\ge 3$ and let $T\in \Clt{X}$ be an element of order $N$. 

\begin{step}
By \cite{Kawamata:bF} we have the inequality
\begin{equation}
\label{eq:r}
0<-K_X\cdot c_2(X)= 24-\sum_{P\in \B} \frac{r_P-1}{r_P}.
\end{equation}
This produces a
finite (but huge) number of possibilities for the basket $\B(X)$ and the number
$-K_X\cdot c_2(X)$. 
\end{step} 

\begin{step}
Theorem~\ref{thm:index} implies that $q\in \{3,\dots,11, 13,17,19\}$.
In each case we compute $A_X^3$ by the formula
\begin{equation*}
A_X^3=\frac{12}{(q-1)(q-2)}\Bigl(
1-\frac{A_X\cdot c_2(X)}{12}+\sum_{P\in B} c_P(-A_X)
\Bigr)
\end{equation*}
(see \cite{Suzuki-2004}), where $c_P$ is the correction term in the
orbifold Riemann-Roch formula \cite{Reid:YPG}. The number $rA_X^3$
must be a positive integer by Theorem~\ref{thm:index}\ref{thm:indexc}.
\end{step} 

\begin{step}
On this step we can use an improved version of Bogomolov--Miyaoka inequality
\cite{Liu-Liu:kawamata--miyaoka}
instead of the one used in \cite{Kawamata:bF} and \cite{Suzuki-2004}. Thus we have
\begin{equation*}
(-K_X)^3\ 
\begin{cases}
< 3(-K_X)\cdot c_2(X)&\text{if $\qQ(X)\neq 4,\, 5$}, 
\\[1em]
\le \frac {25}8(-K_X)\cdot c_2(X)&\text{otherwise}.
\end{cases}
\end{equation*}
This removes a lot of possibilities.
\end{step} 

\begin{step}
In a neighborhood of each point $P\in X$ we can write 
$A_X\sim l_PK_X$, where $0\le l_P<r_P$.
There is a finite number of possibilities for the collection $\{(l_P)\}$.
If $\qW(X)=\qQ(X)$, then $\gcd(q,r)=1$ by Theorem~\ref{thm:index}. 
In this case the numbers $l_P$ are uniquely determined by $1+ql_P\equiv 0\mod r_P$ because $K_X+qA_X\sim 0$.
\end{step} 

\begin{step}
Similarly, a neighborhood of each point $P\in X$ we can write 
$T\sim l_P'K_X$, where $0\le l_P'<r_P$.
The collection $\{(l_P')\}$ and the number $N$ satisfy the following properties:
\[
\chi(X,\, \OOO_X(NT))=1\qquad \text{and}\qquad \chi(X,\, \OOO_X(jT))=0\quad \text{for $j=1,\dots,N-1$}
\]
(by the Kawamata--Viehweg vanishing). Thus we obtain a finite number of possibilities for~$\{(l_P')\}$ and~$N$.
\end{step} 

\begin{step}
Finally, applying Kawamata--Viehweg vanishing we obtain
\begin{equation*}
\chi(X,\, \OOO_X(mA_X+jT))=\h^0(X,\, \OOO_X(mA_X+jT))=0.
\end{equation*}
for $-q<m<0$ and $0\le j<n$. Again, we check this condition using
orbifold Riemann-Roch and remove a lot of possibilities.
\end{step}

\begin{step}
We obtain a list of collections $\left( q, \B(X), A_X^3, \{(l_P)\}, n \{(l_P')\}\right)$.
In each case we compute $\g(X)$ and $\h_X(t,\sigma)$ by using the orbifold Riemann-Roch theorem.
For example, 
\begin{equation}
\label{eq:RR:g}
\g(X)=
-\frac 12 K_X^3+
1-\sum_{P\in \B} \frac{b_P(r_P-b_P)}{2r_P}. 
\end{equation}
\end{step}

\newcommand{\etalchar}[1]{$^{#1}$}
\def\cprime{$'$}


\end{document}